\documentclass[a4paper,11pt]{article}

\usepackage[latin1]{inputenc}				
\usepackage[english]{babel}					
\usepackage{indentfirst}						
\usepackage{color}								

\usepackage{datetime} 

\usepackage{hyperref}
\hypersetup{colorlinks,linkcolor=red,citecolor=blue,urlcolor=green}

\usepackage{graphicx}
\usepackage{subfigure}

\usepackage[top=2.7cm, bottom=3.8cm, left=2.8cm, right=2.8cm]{geometry}		

\usepackage{amsmath}
\usepackage{amsthm}							
\usepackage{amsfonts}							
\usepackage{amssymb}							

\newtheorem{definition}{Definition}[section]				
\newtheorem{theorem}[definition]{Theorem}				
\newtheorem{proposition}[definition]{Proposition}
\newtheorem{corollary}[definition]{Corollary}

\theoremstyle{definition}											
\newtheorem{remark}[definition]{Remark}

\newcommand{\N}{\mathbb{N}}																				
\newcommand{\cN}{\mathcal{N}}																				%
\newcommand{\Z}{\mathbb{Z}}																					
\newcommand{\R}{\mathbb{R}}																					
\newcommand{\F}{\mathcal{F}}																					
\renewcommand{\H}{\mathcal{H}}																				
\renewcommand{\S}{\mathcal{S}}																				
\newcommand{\bE}{\mathbb{E}}																				
\newcommand{\E}{\mathcal{E}}																					
\newcommand{\abs}[1]{\lvert#1\rvert }																		
\newcommand{\Abs}[1]{\left\lvert#1\right\rvert }														
\newcommand{\norm}[1]{\lVert#1\rVert}																	
\newcommand{\Norm}[1]{\left\lVert#1\right\rVert}													
\newcommand{\half}{\frac{1}{2}}																				
\newcommand{\scalar}[2]{\left\langle #1 \big| #2 \right\rangle}								
\newcommand{\goesto}{\rightarrow}																			

\newcommand{\ti}{{t_{i}}}
\newcommand{\tip}{{t_{i+1}}}
\newcommand{\tni}{{t^n_{i}}}
\newcommand{\tnip}{{t^n_{i+1}}}

\newcommand{\wttnip}{{\wt t^n_{i+1}}}
\newcommand{\ip}{{i+1}}
\newcommand{\Hip}{H_{i+1}}

\newcommand{\hip}{h_{i+1}}
\newcommand{\hnip}{h^n_{i+1}}
\newcommand{\wthnip}{\wt h^n_{i+1}}
\newcommand{\hmax}{h_{\text{max}}}

\newcommand{\wh}[1]{\widehat{#1}}
\newcommand{\wt}[1]{\widetilde{#1}}
\newcommand{\wb}[1]{\overline{#1}}

\newcommand{\GMonGr}{(GMonGr) }
\newcommand{\MonY}{(MonY) }
\newcommand{\LipZ}{(LipZ) }
\newcommand{\GrowthY}{(GrY) }
\newcommand{\RegY}{(RegY) }
\newcommand{\kDS}{(k-dSn) }
\newcommand{\kdSn}{(k-dSn) }
\newcommand{\kdSc}{(k-dSc) }
\newcommand{\fzero}{($f(0,0)=0$) }%

\newcommand{\TimeDiscScheme}{\eqref{equation---reference--ATS.scheme}}

\title{Adapted time-steps explicit scheme for monotone BSDEs}
\author{ Arnaud LIONNET \\         
		\footnotesize CMLA, Ecole Normale Supérieure de Cachan}
\date{3/11/2016} 

\begin{document}
\maketitle

\begin{abstract}
	We study the numerical strong stability of explicit schemes for the numerical approximation of the solution to a BSDE where the driver has polynomial growth in the primary variable 
	and satisfies a monotone decreasing condition, and we introduce an explicit scheme with adapted time-steps that guarantee numerical strong stability. 
	We then prove the convergence of this scheme and illustrate it with numerical simulations.
\end{abstract}



{
\hypersetup{colorlinks,citecolor=blue,linkcolor=red} 
\tableofcontents
}

\section{Introduction}

In this paper, we study the qualitative and quantitative properties of explicit numerical methods for backward stochastic differential equations (BSDEs) for a certain class of drivers.
Since the seminal papers of Zhang \cite{Zhang2004} and Bouchard and Touzi \cite{BouchardTouzi2004}, an important literature has 
been concerned with 
the numerical methods for approximating the solution to a nonlinear BSDE 
	\begin{align*}
		Y_t = \xi + \int_t^T f(Y_u,Z_u) du - \int_t^T Z_u dW_u ,
	\end{align*}
where $f$ is the driver, $\xi$ is the terminal condition, $W$ is the underlying, $d$-dimensional Brownian motion, $T>0$ is the time-horizon, and $(Y,Z)$ is the sought solution, adapted to the augmented Brownian filtration, and valued in $\R^k \times \R^{k \times d}$.
Mirroring the development of the general theory for nonlinear BSDEs, this literature has initially focused on the class of drivers that are Lipschitz in both of their $Y$ and $Z$ variables, with square-integrable terminal conditions: \cite{GobetLemorWarin2006}, \cite{BenderDenk2007}, \cite{CrisanManolarakis2012}, \cite{CrisanManolarakis2014}, \cite{ChassagneuxCrisan2014}, \cite{Chassagneux2014}, \cite{BriandLabart2014}, \cite{PagesSagna2015}, \cite{GobetTurkedjiev2016}, among others, have provided analysis of the standard implicit and explicit schemes (also known as backward Euler schemes, Bouchard--Touzi--Zhang schemes, or Euler--BTZ schemes) as well as other time-discretization schemes, and of various methods for approximating the conditional expectations.

In many cases of interest however, the driver has a superlinear growth in one of its variables. 
BSDEs arising in stochastic control problems often have quadratic growth in the $Z$ variable (or control variable), see \cite{BielagkLionnetDosReis2015} or \cite{EliePossamai2016} for recent instances. 
Meanwhile, the BSDEs connected to PDEs of reaction-diffusion type like the Allen--Cahn equation, the FitzHugh--Nagumo equations, the Fisher--KPP equation, etc (see 
\cite{Rothe1984}), tend to have a driver with polynomial growth in the $Y$ variable (or state variable) and to satisfy a so-called monotonicity condition (or one-sided Lipschitz condition).
This is a structure property which states that for all $y,y'$ in the domain $\R^k$ and for all $z \in \R^{k \times d}$,
	\begin{align*}		
		\scalar{  y'-y }{  f(y',z)-f(y,z)  } \le M_y \abs{y'-y}^2 ,
	\end{align*}
where $M_y \in \R$. 
In the case of drivers that are quadratic in the control variable and of bounded terminal condition, Chassagneux and Richou \cite{ChassagneuxRichou2016} recently established the convergence of a fully implicit scheme with modified driver (see also \cite{ImkellerReis2010,Richou2011} for earlier results). 
The case of monotone drivers with polynomial growth in the state variable and terminal conditions with all moments, was studied in \cite{LionnetDosReisSzpruch2015}, where it was proved that the standard implicit scheme converges. The standard explicit scheme however may explode, for unbounded terminal conditions. 
As a remedy, in order to obtain a modified explicit scheme (also referred to as tamed explicit scheme, see \cite{HutzenthalerJentzenKloeden2011}) with the same convergence rate as the standard implicit scheme but the computational cost of an explicit scheme, one solution is to truncate the terminal condition with a truncation that relaxes to the identity as the modulus of the time-discretization grid goes to $0$ (see \cite{LionnetDosReisSzpruch2015}). 
Another is to replace the driver by a modified driver that converges to the original driver as the modulus of the time-discretization grid goes to $0$ (see \cite{LionnetDosReisSzpruch2016}).

\paragraph*{}

However the convergence of the scheme, and its convergence rate, are only asymptotic properties. 
In practice, it can be more interesting to obtain a scheme with a qualitatively good behaviour and a decent error for a low computational budget, rather than having very good results guaranteed only for high-enough computational efforts. 
Generally speaking, we will say that a scheme is numerically stable if it preserves some properties (qualitative or quantitative) enjoyed by the continuous-time BSDE.
Chassagneux and Richou \cite{ChassagneuxRichou2015} thus studied the numerical stability in the sense of preserving the following size estimate: when the driver is truly monotone (i.e. satisfies the property with at least $M_y = 0$), satisfies $f(0,0)=0$ and the terminal condition is bounded, one has $\abs{Y_t} \le \norm{\xi}_\infty$ at all times.
They study for both of the standard schemes (explicit and implicit), in the case of Lipschitz drivers, the precise conditions on the time-steps under which such a bound is preserved. They were also able to cover the standard implicit scheme in the case of polynomial growth drivers.

In this work, were are interested in polynomial growth drivers and the numerical stability properties of the explicit scheme. In particular, we are interested in the preservation of two properties. 
One is the fact that, for strictly monotone drivers (i.e. $M_y<0$) we have in fact for all $t \in [0,T]$ the stronger bound $\abs{Y_t} \le e^{M_y (T-t)} \norm{\xi}_\infty$. We call the preservation of this bound strong (or strict) numerical stability. 
The second is the comparison property, stating that if we have ordered terminal conditions $\xi \le \xi'$, then $Y \le Y'$. As explained in \cite{LionnetDosReisSzpruch2016}, the preservation of this property is crucial when the driver is monotone only on a domain $D \subseteq \R^k$ in which the continuous-time solution is known to remain.

This study naturally leads to an explicit scheme with a non-uniform time-grid, which adapts to both the terminal condition and the driver. This scheme, called Adapted Time-Steps scheme, is numerically stable by design. We also show its convergence. While the total error of a converging, first-order scheme is usually found to be bounded above by some $e^{C \, T} / n$, for some constant $C \ge 0$ and a number $n$ of time-steps (possibly for $n$ large enough), we here aim to obtain an upper bound as good as possible: we can bound the error by some $e^{C^n \, T} / n$, where the constants $C^n$ have a limit $C' < 0$. 
We thus obtain an explicit scheme that allocates the computational effort non-uniformly over $[0,T]$, with more dates toward $T$, and so guarantees good qualitative and behaviour as well as having good quantitative properties.

\paragraph*{}
The next section is devoted to stating precisely the conditions governing the class of drivers we consider, recalling known results for the associated (continuous-time) BSDEs, and stating the adapted time-steps (ATS) grid. 
Section \ref{section---numerical.stability} is then concerned with the analysis of the numerical stability of the scheme, both in terms of size bounds and comparison property. 
The convergence is then proved in section \ref{section---convergence.of.the.scheme}. 
After that we consider in sections \ref{section---extension.to.unbounded.term.cond.} and \ref{section---extensions.to.overall-monotone.drivers} two extensions of our working assumptions and show how to adapt the scheme in these situations. 
The first one is when the terminal condition is unbounded, in which case we can do a fading truncation of the terminal condition. 
The second one is when the driver is ``monotone'' in the wider sense, with $M_y \ge 0$, but satisfies a monotone growth condition $\scalar{y}{f(y,z)} \le M^a_y \abs{y}^2$ with $M^a_y < 0$ when $y$ is away from the origin.
In section \ref{section---numerical.simulations} we provide numerical simulations to illustrate the qualitative and quantitative behaviour of the ATS scheme.

\section{Preliminaries}		\label{section---preliminaries}

We work on a filtered probability space $(\Omega,\F=(\F_t)_{t \in [0,T]},P)$ carrying a standard Brownian motion $W$, whose augmented natural filtration is $\F$.


For fixed $x_0 \in \R^d$ we consider the system of forward and backward SDEs, for $0 \le t \le T$,  
	\begin{align}
		\label{equation---reference--SDE}		
		X_t &= x_0 + \int_0^t \mu(s,X_s) ds + \int_0^t \sigma(s,X_s) dW_s			\\
		\label{equation---reference--BSDE}
		Y_t &= g(X_T) + \int_t^T f(Y_u, Z_u) du - \int_t^T Z_u dW_u .
	\end{align}

In all generality, the driver of the BSDE can depend on the time $u$ and $X_u$. However, as can be seen from \cite{LionnetDosReisSzpruch2015}, 
this dependence can be dealt with in the numerical analysis but makes for much heavier computations without bringing much interesting results and behaviour in the picture. So we only assume a dependence of $f$ in the variables $Y$ and $Z$.

\subsubsection*{Assumptions.}

\paragraph*{Forward equation.} 
We take $\mu : [0,T] \times \R^d \rightarrow \R^d$ and $\sigma : [0,T] \times \R^d \rightarrow \R^{d \times d}$ to be $\half$-Hölder in their time-variable and Lipschitz in space-variable. 

\paragraph*{Backward equation.} 
We assume $g : \R^d \rightarrow \R^k$ to be Lipschitz and bounded, so that the terminal condition $\xi = g(X_T)$ is bounded and $\Norm{\xi}_\infty \le \Norm{g}_\infty$. 
The function $f : \R^k \times \R^{k \times d} \rightarrow \R^k$ is truly monotone in $y$, locally Lipschitz in $y$, with at most polynomial growth of degree $m \in \N^*$, and Lipschitz in $z$. 
Specifically, we assume the following.

	\begin{itemize}
		\item[] \MonY : $f$ is monotone in $y$ with monotonicity constant $M_y \le 0$. For all $y,y',z$,
				\begin{align*}
					\scalar{y'-y}{f(y',z)-f(y,z)} \le M_y \abs{y'-y}^2 .
				\end{align*}
		\item[] \RegY : $f$ has the following $y$-regularity, with constant $L_y \ge 0$. For all $y,y',z$,
				\begin{align*}
					\Abs{ f(y',z)-f(y,z) } \le L_y \big( 1 + \abs{y'}^{m-1} + \abs{y}^{m-1} \big) \abs{y'-y} .
				\end{align*}
		\item[] \LipZ : 	$f$ is Lipschitz in $z$ with constant $L_z \ge 0$. For all $y,z,z'$,	
				\begin{align*}
					\Abs{ f(y,z')-f(y,z) } \le L_z \abs{z'-z} .
				\end{align*}
		\item[] \fzero : we have $f(0,0)=0$.
	\end{itemize}

As explained in section \ref{section---extensions.to.overall-monotone.drivers}, the assumption \fzero is mostly a convenient centering.

In the scalar case, where $k=1$, one sees that what is standarly called ``monotonicity condition'', with a general $M_y \in \R$, in fact means that the slopes of all the chords of $f$ are bounded above: for all $y',y \in \R$ and $z \in \R^{1 \times d}$ we have
	\begin{align*}
		\frac{f(y',z)-f(y,z)}{y'-y} \le M_y .
	\end{align*}
The function $f$ can in fact be decomposed as $f = f^d + f^l$ where $f^d$ is a decreasing function and $f^l$ is a Lipschitz function, see \cite{ChassagneuxJacquierMihaylov2014}. 
In particular, it could be locally strictly increasing and locally strictly decreasing, but not be monotone over $\R$ in the usual sense. 
When $M_y \le 0$ as in \MonY, the condition implies that $f$ is decreasing, and in particular monotone. We will say in this paper that the function truly monotone. 
When $M_y < 0$, $f$ is strictly decreasing. We will say in this paper that the function strictly monotone. 

These assumptions are assumed to hold in all of the paper, except where otherwise stated.

\subsubsection*{Properties of the continuous-time solution and stability assumption}


Under the assumptions above, the forward and backward SDEs \eqref{equation---reference--SDE}-\eqref{equation---reference--BSDE} are well-posed. 
More precisely, from \cite{Pardoux1999}, \cite{BriandCarmona2000}, and sections 2 and 3 of \cite{LionnetDosReisSzpruch2015}, we have the following results.
Firstly, \eqref{equation---reference--SDE}-\eqref{equation---reference--BSDE} has a unique solution $(X,Y,Z)\in \S^p\times \S^p \times \H^p$ for any $p\geq 2$, and it satisfies 
		\begin{align}		\label{equation---reference--bounds.on.X.Y.Z} 
			\Norm{X}_{\S^p}^{p} \le C_p 
			\qquad \text{and} \qquad 
			\Norm{Y}_{\S^p}^{p} + \Norm{Z}_{\H^p}^{p} \le C_p  \Norm{\xi}_{L^p}^{p}.
		 \end{align}
We recall that $\S^p$ is the space of continuous semimartingales $X$ such that $\Norm{X}_{\S^p}^{p} = E\big[ \sup_{0\le t \le T} \abs{X_t}^p \big] <+\infty$, and $\H^p$ is the space of progressively-measurable processes $Z$ such that $\Norm{Z}_{\H^p}^p = E\big[ \big( \int_0^T \abs{Z_t}^2 dt \big)^{p/2} \big] <+\infty$.

Secondly, let $\pi = (t_i)_{i = 0 \ldots N}$ be any time-grid (or partition, or subdivision) of the interval $[0,T]$, that is to say such that $0 = t_0 < t_1 < \ldots < t_N = T$, 
and denote its modulus by $\abs{\pi} = \max\{ \hip | i = 0 , \ldots, N-1 \}$, where $\hip = t_\ip - t_i$.
Define the random variables $(\wb{Z}_\ti)_{i=0,\cdots,N-1}$ by
			\begin{align*}
				\overline{Z}_\ti =  E_i\bigg[ \frac{1}{\hip} \int_\ti^\tip Z_u du \bigg].
			\end{align*}
For any $p\geq 2$ there exists a positive constant $C_p$ (independent of $\pi$) such that we have 
	\begin{align}		\label{equation---reference--path.regularity.Y.Z}
		\begin{aligned}
			&\mathrm{REG}_{Y,p}(\abs{\pi}) := \sup \Big\{ E\big[ \Abs{Y_s - Y_t}^p \big]^{{1}/{p}}	\ \Big| \  s,t\in[0,T] \text{ and } \abs{t-s}\le \abs{\pi} \Big\} \leq C_p \abs{\pi}^{1/2} , \\
			&\mathrm{REG}_{Z,2}({\pi}) := E\bigg[ \sum_{i=0}^{N-1} \int_\ti^\tip \abs{Z_t - \overline{Z}_\ti}^2 dt \bigg]^{1/2} \le C_2 \abs{\pi}^{1/2}.
		\end{aligned}
	\end{align}

\paragraph*{}
Additionally, since $\xi=g(X_T)$ is bounded, the solution to the continuous-time BSDE satisfies the followings bounds.
%
In dimension $k=1$, using the linearization technique (see for instance \cite{BriandElie2013}), we have for all $t \in [0,T]$
	\begin{align*}
		\Norm{Y_t}_\infty \le e^{M_y (T-t)} \Norm{\xi}_\infty .	
	\end{align*}
Clearly, as soon as $M_y \le 0$, as is assumed here in \MonY, we have, $\Norm{Y}_{\S^\infty} \le \Norm{\xi}_\infty$.	
Recall that $\S^\infty$ is the space of bounded continuous semimartingales, with norm $\Norm{Y}_{\S^\infty} = \sup_{0 \le t \le T} \Norm{Y_t}_\infty$.

In the general case of a dimension $k \in \N^*$, we have (see \cite{LionnetDosReisSzpruch2015})
	\begin{align*}
		\abs{Y_t}^2 + \half E_t\bigg[ \int_t^T \abs{Z_u}^2 du \bigg] \le e^{2 (M_y + L_z^2) (T-t)} \Norm{\xi}_\infty ,
	\end{align*}
so in particular 
	\begin{align*}
		\Norm{Y_t}_\infty \le e^{(M_y + L_z^2) (T-t)} \Norm{\xi}_\infty .
	\end{align*}
Therefore, a sufficient condition for having $\Norm{Y}_{\S^\infty} \le \Norm{\xi}_\infty$ is that $L_z^2 \le -M_y$. This condition, which will be called \kdSc from now on ($k$-dimensional stability, continuous case), is assumed to hold in the rest of the paper. Note that in the limiting case $M_y=0$, \kdSc implies that we require $L_z=0$ : no dependence in $z$.

A focus of this work and a main goal of the ATS scheme is to preserve these bounds in the numerical scheme.
Unless otherwise stated (see section \ref{section---extensions.to.overall-monotone.drivers}), we assume $M_y < 0$, as this is where the scheme is useful.

\subsubsection*{Time-discretization, for a given grid}

Let $\pi = (t_i)_{i = 0 \ldots N}$ be any  time-grid.
We use this time-grid to discretize the dynamics of the forward and backward SDEs \eqref{equation---reference--SDE}-\eqref{equation---reference--BSDE} as follows.
As we are not concerned with obtaining high-order schemes here, we simply use the explicit Euler--Maruyama scheme for $X$. 
It is initialized with $X_0 = x_0$ and then, for $i = 0$ up to $N-1$, writing $\Delta W_\tip = W_\tip - W_\ti$,
	\begin{align}		\label{equation---reference--X.scheme-time.discretization.only}
		X_\ip = X_i + \mu(\ti,X_i) \hip + \sigma(\ti,X_i) \Delta W_\tip .
	\end{align}
For $(Y,Z)$, we use a form of the explicit Bouchard--Touzi--Zhang (BTZ) scheme over the grid $\pi$. 
Denote $E_i=E[\cdot|\F_\ti]$.
It is initialized with $Y_N = \xi^N := g(X_N)$ and then, for $i=N-1$ down to $0$, $Y_\ip$ is used to compute $(Y_i,Z_i)$ as
	\begin{align}		\label{equation---reference--ATS.scheme}
		\left\{\begin{aligned}
			Y_i &= E_i\Big[ Y_\ip + f(Y_\ip,Z_i) \hip \Big]											\\
			Z_i &= E_i\Big[ Y_\ip \Hip^* \Big] .
		\end{aligned}\right.
	\end{align}

Here, $\Hip = T^{R^\pi_i}(\Delta W_\tip) / \hip$ and $T^{R^\pi_i}$ is the projection in the $L^\infty$-norm of $\R^d$, i.e. component-wise, on the ball of radius $R^\pi_i > 0$.
We take $R^\pi_i$ to be a function of the time-step $\hip$ : $R^\pi_i = R(\hip)$, where $R : h \mapsto \sqrt{h} r(h)$ and $r(h) = \sqrt{2} \ln(1/h)$.
We impose that $\abs{\pi} \le \hmax = 0.95$, so that $r(h) > 0$ for any $h = \hip \le \abs{\pi} \le \hmax < 1$.
Let us also define $\Lambda(h) = E[T^{r(h)}(G)^2]$, for any scalar random variable $G \sim \cN(0,1)$, where $T^r$ is the projection in $\R^1$ on the ball of radius $r$.
As $0 < h \le \hmax$, we have $\Lambda_{\text{min}} = \Lambda(\hmax) \le \Lambda(h) < \Lambda(0^+) = 1$. 
Denoting by $\Lambda_i = \Lambda(\hip)$ and by $I_d$ the identity of  $\R^{d \times d}$, we thus have
	\begin{align*}
		E\big[ \Hip \Hip^* \big] = \frac{\Lambda_i I_d}{\hip}
		\qquad \text{and} \qquad 
		E\big[ \abs{\Hip}^2 \big] = \frac{\Lambda_i d}{\hip} .
	\end{align*}
It results from the choice of $R$ that there exists $C^H \ge 0$ such that for all $\pi$ : $\abs{\pi} \le \hmax$, 
	\begin{align}		\label{equation---reference--bound.on.error.from.truncating.H}
		\max_{i = 0 \ldots N-1} E\left[ \Abs{ \frac{\Delta W_\tip}{\hip} - \Hip }^2 \right]	 \le C^H .
	\end{align}
{See the appendix for a proof of this.} 

	\begin{remark}
		The above scheme without the truncation of the Brownian increment $\Delta W_\tip$ would be the standard explicit BTZ scheme. 
		As will be seen later, requiring $\Hip$ to be bounded is crucial to ensure some numerical stability for the scheme (the comparison property, to be precise).
		
		Also, it is argued in \cite{LionnetDosReisSzpruch2016} that given that $Y_i$ is defined as $E_i\big[ Y_\ip + f(Y_\ip,Z_i) \hip \big]$ it would be in a sense more natural to set
		$Z_i = E_i\big[ \big( Y_\ip + (1-\theta') f(Y_\ip,Z_i) \hip \big) \Hip^* \big]$ with $\theta'=0$. 
		However the precise value of  $\theta'$ does not change the order of the scheme nor does it change much the estimations that will follow. 
		But $\theta'<1$ leads to an implicit equation in $Z_i$. 
		So we take $\theta'=1$ for simplicity.
	\end{remark}

	\begin{remark}[Notation]
		Sometimes, we emphasize that these approximations are induced by the grid $\pi$ by denoting them $(X^\pi_i)_{i = 0 \ldots N}$ and $(Y^\pi_i,Z^\pi_i)_{i = 0 \ldots N}$. 
		In particular, if the grid $\pi^n$ is characterized by a numerical parameter $n$ as will be the case below, we may write $X^n_i$, $Y^n_i$ and $Z^n_i$.
		Note that the discrete indexing by $i \in \{0,\ldots,N\}$ prevents the confusion of $X_i$, $Y_i$ and $Z_i$ with the quantities $X_\ti$, $Y_\ti$ and $\wb{Z}_\ti$.
	\end{remark}

	\begin{remark}
		Given that we have assumed $g$ to be Lipschitz and that the Euler--Maruyama scheme is known to have strong order $1/2$, 
		there exists a constant $C^X$ such that
			\begin{align}		\label{equation---reference--error.bound.on.the.terminal.condition}
				E[\abs{\xi-\xi^N}^2] \le C^X \abs{\pi} .
			\end{align}
	\end{remark}

\subsubsection*{Adapted time steps}

We now describe the adapted time steps used for the scheme. 

\paragraph*{One-dimensional case.}
Let $h_0(L_z) = \sup \{ h \in \ ]0,\hmax] \ \big| \ R(h) \le \frac{1}{3 L_z} \}$. Note that since $R(h) = \sqrt{h} \sqrt{2} \ln(1/h)$ has limit $0$ as $h \goesto 0$ and $h \goesto 1$, $R(\cdot)$ is bounded on $]0,1]$, hence on $]0,\hmax]$. So if $L_z$ is small enough, we have  $h_0(L_z) = \hmax$. If $L_z=0$, we define $h_0(L_z) = \hmax$.

For a given number $n \in \N^*$, we define the time grid $\pi^n = \{ t_i^n \ |\ i = 0 \ldots N(n) \}$ by setting $t^n_N = T$ and,
for decreasing $i$'s, setting $t^n_i = t^n_\ip - h^n_\ip$ where 
	\begin{align}		\label{equation---reference--definition.of.hnip.1dim}
		h^n_\ip := \min \left\{ \frac{1}{3L_y \big( 1 + e^{(m-1) M_y (T-t^n_\ip)} \Norm{\xi}_\infty^{m-1} \big)} \ , \ h_0(L_z) \ , \ \hmax \ , \ \tnip \ , \ \frac{T}{n} \right\} .
	\end{align}
The first two requirements are there to ensure that the scheme is numerically stable in the sense of preserving some size-bounds on the output. 
Note for later that $h \le h_0(L_z)$ implies $\hnip L_z \Norm{\Hip}_\infty \le 1/3$. The third requirement is only to ensure that $R(\hnip) > 0$ and the fourth to ensure that $\tni \ge 0$ at each step. The last one guarantees that $\abs{\pi^n} \le T/n \goesto 0$ as $n \goesto +\infty$.

\paragraph*{Multi-dimensional case.}
As the condition for numerical stabilty is different when $k>1$, we construct the time-grid $\pi^n$ as above but with time-steps of sizes
	\begin{align}		\label{equation---reference--definition.of.hnip.multidim}
		h^n_\ip := \min \left\{ \frac{1}{4} \ \frac{-M_y}{(d+1) L_y^2 ( 1 + e^{(m-1) \frac{M_y}{2} (T-\tnip)} \Norm{\xi}^{2(m-1)}_\infty ) } \ , \ \frac{1}{8} \ \frac{1}{2(d+1) L_z^2 } \ ,\ \hmax \ , \ \tnip \ ,\ \frac{T}{n} \right\} .
	\end{align}
Additionally, the continuous-time stability condition \kdSc has to be strengthened, and we will require for the numerical approximation that $L_z^2 \le \frac14 (-M_y)$, condition denoted by \kDS from now on ($k$-dimensional stability, numerical version).

As $i$ decreases from $N-1$ to $1$, the time-step $\hnip$ increases (in a large sense), and then $h^n_1 = t^n_1 > 0$ can be arbitrarily close to $0$.
In the limit $n \goesto +\infty$, the partition ``degenerates'' into the uniform partition. 

Remark that, unlike what would be possible with an ODE, the adapted time-steps $h^n_\ip$ cannot be computed online here, i.e. as the algorithm computing $(Y_i,Z_i)$ progresses backward in time, and cannot make use of $\Norm{Y^n_\ip}_{\infty}$ for computing $\hnip$. 
Indeed, the computation of $Y_N=\xi^N$ and of the conditional expectations involved in defining the $Y_i$'s and $Z_i$'s requires having already approximated the forward process $X$ at the points of the partition $\pi^n$.
So $\pi^n$ has to be computed ahead of the computation of the $X_i$'s and $(Y_i,Z_i)$'s.. 
As will be seen later, it implicitly makes use of a priori bounds for $\Norm{Y^n_\ip}_{\infty}$.

\section{Numerical stability}			\label{section---numerical.stability}

In this section, we study the numerical stability of the scheme. More precisely, we first look at the conditions for an explicit scheme to preserve in discrete time the bound $\Norm{Y_\ti}_\infty \le \Norm{\xi}_\infty$ from the continuous-time dynamics.
We in fact look at the sharper bound $\Norm{Y_\ti}_\infty \le e^{M_y (T-\ti)} \Norm{\xi}_\infty$. 
We are also interested, in the scalar case $k=1$, in the preservation of the comparison property. 

\subsection{The one-dimensional case}

\subsubsection{Size estimate for Y}

Here we consider an arbitrary partition $\pi$ and focus on a single time-step $[\ti,\tip]$, of size $\hip = \tip-\ti$.

	\begin{proposition}		\label{proposition---size.estimate.one.step.via.linearization.1d}
		Let $k=1$, and let $Y_\ip$ be a bounded input for the scheme \eqref{equation---reference--ATS.scheme}. 
		Define 
			\begin{align*}
				h^0_\ip(\norm{Y_\ip}_\infty) := \min \left\{ \frac{1}{3L_y \big( 1 + \Norm{Y_\ip}_\infty^{m-1} \big)} \ , \ h_0(L_z) \ , \ \hmax \right\} .
			\end{align*}
		If $\hip \le h^0_\ip(\norm{Y_\ip}_\infty)$, then the output $(Y_i,Z_i)$ of the scheme is such that
			\begin{align}		\label{equation---size.estimate.one.step.infty.norm}
				\Norm{Y_i}_\infty \le e^{M_y \hip} \Norm{Y_\ip}_\infty \le \Norm{Y_\ip}_\infty .
			\end{align}
	\end{proposition}

	\begin{proof}
		We use a linearization technique. 
		Recall that $f(0,0)=0$ and define
			\begin{align*}
				\beta_\ip &= \frac{f(Y_\ip,Z_i)-f(0,Z_i)}{Y_\ip} 
				\qquad \text{and} \\
				\gamma_i &= 1_{\{Z_i \neq 0\}} \frac{f(0,Z_i)-f(0,0)}{\abs{Z_i}^2} Z_i^* . 
			\end{align*}
		Note that since $k=1$, $Z_i \in \R^{1 \times d}$. Then $Z_i^*$, or $\gamma_i$, meant to be an approximate derivative of $f$ with respect to $Z$, lies in 
		$\R^{d \times 1} \approx (\R^{1 \times d})^*$, the dual of $\R^{1 \times d}$, and acts on $Z \in \R^{1 \times d}$ by $\gamma_i \cdot Z_i = \sum_{l=1}^d \gamma_i^l Z_i^{1,l}$, 
		i.e. $\gamma_i \cdot Z_i = Z_i \gamma_i$.
		Writing
			\begin{align*}
				f(Y_\ip,Z_i) = 0 + \beta_\ip Y_\ip + \gamma_i \cdot Z_i ,
			\end{align*}
		recalling that $Z_i = E_i[Y_\ip \Hip^*]$ and $\gamma_i$ is $\F_i$-measurable, 
		we have
			\begin{align*}
				Y_i &= E_i\Big[ Y_\ip + \beta_\ip Y_\ip \hip + \gamma_i \cdot Z_i \hip \Big] 															\\
					&= E_i\Big[ Y_\ip \big(1 + \beta_\ip \hip \big) + \gamma_i \cdot E_i[Y_\ip \Hip^*] \hip \Big] 							\\
					&= E_i\Big[ Y_\ip \big(1 + \beta_\ip \hip \big) + \gamma_i \cdot Y_\ip \Hip^* \hip \Big]									\\
					&= E_i\Big[ Y_\ip \big(1 + \beta_\ip \hip + \gamma_i \cdot \Hip^* \hip \big) \Big]											\\
					&= E_i\Big[  Y_\ip\ \E_\ip \Big].
			\end{align*}
		From \RegY and \LipZ we have
			\begin{align*}
				\abs{\beta_\ip} \le L_y (1 + \abs{Y_\ip}^{m-1}) \le L_y (1 + \Norm{Y_\ip}_\infty^{m-1})		
				\qquad \text{and} \qquad
				\abs{\gamma_i} \le L_z .
			\end{align*}
		So 
			\begin{align*}
				\E_\ip \ge 1 - \hip L_y (1 + \Norm{Y_\ip}_\infty^{m-1}) - \hip L_z \Norm{\Hip}_\infty \ge \frac13 > 0 ,
			\end{align*}
		since $\hip$ is assumed such that 
			\begin{align*}
				\hip \le \frac{1}{3L_y \big( 1 + \Norm{Y_\ip}_\infty^{m-1} \big)}
				\quad \text{and} \quad 
				\hip \le h_0(L_z) \Rightarrow \hip \Norm{\Hip}_\infty \le R(\hip) \le \frac{1}{3L_z} .
			\end{align*}
		From the positivity of $\E_\ip$ it results that
			\begin{align*}
				\abs{Y_i} 
					&\le E_i\Big[ \E_\ip \ \abs{Y_\ip} \Big]																					\\
					&\le E_i\Big[ 1 + \beta_\ip \hip + \gamma_i \cdot \Hip^* \hip \Big] \Norm{Y_\ip}_\infty 			\\
					& =  E_i\Big[ 1 + \beta_\ip \hip \Big] \Norm{Y_\ip}_\infty 													\\
					&\le e^{M_y \hip} \Norm{Y_\ip}_\infty 
						\le \Norm{Y_\ip}_\infty ,
			\end{align*}
		since $\gamma_i$ is $\F_i$-measurable, $E_i[\Hip]=0$, $1+x \le e^x$, and $\beta_\ip \le M_y$ by \MonY. 
		The last inequality follows from the fact that $M_y \le 0$.
	\end{proof}

We now consider, for a fixed $n \in \N^*$, the ATS grid $\pi^n$ and the associated approximations $(Y^n_i,Z^n_i)_{i = 0 \ldots N}$.

	\begin{corollary}		\label{proposition---size.estimate.all.steps.via.linearization.1d}
		Let $k=1$. We have, for all $i \in \{0, \ldots, N \}$,
			\begin{align*}
				\Norm{Y^n_i}_\infty \le e^{M_y (T-\tni)} \Norm{\xi^N}_\infty .
			\end{align*}
	\end{corollary}

	\begin{proof} 
		The result follows from a backward induction. The estimate is clearly true for $i=N$. 
		Next, assume it is true for $i+1$, for $i \in \{0, \ldots N-1\}$. 
		Taking the power $m-1$ of it we have $\Norm{Y^n_\ip}_\infty^{m-1} \le e^{(m-1) M_y (T-\tnip)} \Norm{\xi^N}_\infty^{m-1}$, and therefore
			\begin{align*}
				h^n_\ip 
					&= \min \left\{ \frac{1}{3L_y \big( 1 + e^{(m-1) M_y (T-\tnip)} \Norm{\xi}_\infty^{m-1} \big)} \ , \ h_0(L_z) \ , \ \hmax \ , \ \tnip \ , \ \frac{T}{n} \right\} 		\\
						&\le \min \left\{ \frac{1}{3L_y \big( 1 + \Norm{Y_\ip}_\infty^{m-1} \big)} \ , \ h_0(L_z) \ , \ \hmax  \right\} 
							= h^0_\ip(\Norm{Y_\ip}_\infty) .
			\end{align*}
		Proposition \ref{proposition---size.estimate.one.step.via.linearization.1d} then guarantees that 
			\begin{align*}
				\Norm{Y^n_i}_\infty \le e^{M_y \hnip} \Norm{Y^n_\ip}_\infty \le e^{M_y \hnip} \ e^{M_y (T-\tnip)} \Norm{\xi^N}_\infty =  e^{M_y (T-\tni)} \Norm{\xi^N}_\infty ,
			\end{align*}
		as wanted.
	\end{proof}

In \cite{ChassagneuxRichou2015}, Chassagneux and Richou studied the notion of numerical stability associated with the following property of the continuous-time BSDE:
when $M_y \le 0$ and $\xi \in L^\infty$, the solution satisfies $\Norm{Y_t}_\infty \le \Norm{\xi}_\infty$ for all $t$.
They determined the $h_0$'s under which, for Lipschitz drivers, the individual steps of the implicit and explicit BTZ schemes are numerically stable and, for monotone drivers, the implicit BTZ scheme is numerically stable. 
Our result above determines the $h_0$ under which, for monotone drivers of polynomial growth, the steps of the explicit BTZ scheme are numerically stable in the same sense. 
In contrast with the implicit case however (see propositions 2.1 and 2.2 in \cite{ChassagneuxRichou2015}), the condition for an individual time-step of the explicit scheme to be stable has to depend on $\Norm{Y_\ip}_\infty$ and the polynomial growth $m$ of the driver, rather that only the constants related to its regularity ($M_y,L_z$, etc).

Our result also allows to conclude about the numerical stability in a stronger sense, associated the bound $\Norm{Y_t}_\infty \le e^{M_y (T-t)} \Norm{\xi}_\infty < \Norm{\xi}_\infty$, 
when $M_y < 0$. 
We note that the discretized dynamics here is damped at the same rate $M_y$ as the continuous dynamics. 
In the multidimensional case and the error estimates below, it will not be possible to obtain this: the rate will still be negative but only in an interval $[M_y,M_y/2] \subseteq \ ]-\infty,0[ \,$.

\subsubsection{A distance estimate for $Y$ and comparison theorem}

We now consider inputs $\wh{Y}_\ip$ and $Y_\ip$, and the respective outputs $(\wh{Y}_i,\wh{Z}_i)$ and $(Y_i,Z_i)$, from one step of the scheme \TimeDiscScheme. We denote generically by $\delta x$ the quantity $\wh{x}-x$.

	\begin{proposition}		\label{proposition---stability.estimate.one-step.via.linearization.1d} 
		Let k=1 and assume that $\hip \le h^0_\ip(\norm{\wh Y_\ip}_\infty,\norm{Y_\ip}_\infty)$, where  
			\begin{align*}
				h^0_\ip(\norm{\wh Y_\ip}_\infty,\norm{Y_\ip}_\infty) 
					:= \min \left\{ \frac{1}{3 L_y (1 + \norm{\wh{Y}_\ip}_\infty^{m-1} + \norm{Y_\ip}_\infty^{m-1})} \ , \ h_0(L_z) \ , \ \hmax \right\} .
			\end{align*}		
			
		We have the representation
			\begin{align*}
				\delta Y_i = E_i\Big[ \E_\ip \, \delta Y_\ip \Big] 
			\end{align*}
		for some random variable $\E_\ip$ which is strictly positive and bounded.
		Furthermore, we have
			\begin{align*}
				\norm{\delta Y_i}_\infty \le e^{ M_y \hip } \norm{\delta Y_\ip}_\infty  
				\qquad \text{and} \qquad 
				\abs{\delta Y_i}^2 \le e^{ L_z^2 \hip } E_i\Big[ \abs{\delta Y_\ip}^2 \Big] .
			\end{align*}						
	\end{proposition}

	\begin{proof}
		We have
			\begin{align*}
				\delta Y_i &= E_i\Big[ \delta Y_\ip + \big( f(\wh{Y}_\ip,\wh{Z}_i) - f(Y_\ip,Z_i) \big) \hip \Big]	,	\\
				\delta Z_i &= E_i\big[ \delta Y_\ip \Hip^* \big] .
			\end{align*}
		Defining
			\begin{align*}
				\beta_\ip = \frac{ f(\wh{Y}_\ip,\wh{Z}_i) - f(Y_\ip,\wh{Z}_i) }{ \delta Y_\ip }
				\quad \text{and} \quad 
				\gamma_\ip = 1_{\{\delta Z_i \neq 0\}} \ \frac{ f(Y_\ip,\wh{Z}_i) - f(Y_\ip,Z_i) }{ \abs{\delta Z_i}^2 } \delta Z_i^*
			\end{align*}
		we can write
			\begin{align*}
				f(\wh{Y}_\ip,\wh{Z}_i) - f(Y_\ip,Z_i) 
					&= f(\wh{Y}_\ip,\wh{Z}_i) - f(Y_\ip,\wh{Z}_i) +  f(Y_\ip,\wh{Z}_i) - f(Y_\ip,Z_i)		\\
					&= \beta_\ip \delta Y_\ip + \gamma_\ip \cdot \delta Z_i .
			\end{align*}
		Since $\delta Z_i$ is $\F_i$-measurable, we can write
			\begin{align*}
				\delta Y_i
					&= E_i\Big[ \delta Y_\ip + \beta_\ip \delta Y_\ip \hip + \gamma_\ip \cdot \delta Z_i \hip \Big]													\\
					&= E_i\Big[ \delta Y_\ip \big (1 + \beta_\ip \hip \big) + E_i[ \gamma_\ip ] \cdot \delta Z_i \hip \Big]											\\
					&= E_i\Big[ \delta Y_\ip \big (1 + \beta_\ip \hip \big) + E_i[ \gamma_\ip] \cdot E_i\big[ \delta Y_\ip \Hip^* \big] \hip \Big]		\\
					&= E_i\Big[ \delta Y_\ip \big (1 + \beta_\ip \hip \big) +  E_i\big[ \delta Y_\ip E_i[\gamma_\ip] \cdot \Hip^* \hip \big] \Big]		\\
					&= E_i\Big[ \delta Y_\ip \big (1 + \beta_\ip \hip + E_i[\gamma_\ip] \cdot \Hip^* \hip \big) \Big] 												\\
					&= E_i\Big[ \delta Y_\ip \, \E_\ip \Big] ,
			\end{align*}			
		with $\E_\ip = 1 + \beta_\ip \hip + E_i[\gamma_\ip] \cdot \Hip^* \hip$.
		Now, using \RegY and \LipZ we have
			\begin{align*}
				\Abs{\beta_\ip} \le L_y (1 + \norm{\wh{Y}_\ip}_\infty^{m-1} + \norm{Y_\ip}_\infty^{m-1}) 
				\qquad \text{and} \qquad 
				\Abs{E_i[\gamma_\ip]} \le E_i[\Abs{\gamma_\ip}] \le L_z .
			\end{align*}
		Given the assumption on $\hip$, we have
			\begin{align*}
				L_y (1 + \norm{\wh{Y}_\ip}_\infty^{m-1} + \norm{Y_\ip}_\infty^{m-1}) \hip \le \frac13
				\quad \text{and} \quad 
				\hip \le h_0(L_z) \Rightarrow \Norm{\Hip}_\infty \hip \le \frac{1}{3L_z}
			\end{align*}
		so we have 
			\begin{align*}
				\frac53 \ge \E_\ip = 1 + \beta_\ip \hip + E_i[\gamma_\ip] \cdot \Hip^* \hip \ge \frac13 ,
			\end{align*}
		which proves the first claim of the proposition.
		As a consequence of the positivity of $\E_\ip$, and using the fact that $E_i\big[ E_i[\gamma_\ip] \cdot \Hip^* \big] \hip = E_i[\gamma_\ip] \cdot E_i[\Hip^*] \hip  = 0$,
			\begin{align*}
				\Abs{\delta Y_i}
					&\le E_i\Big[ \E_\ip \Abs{\delta Y_\ip} \Big] 																												\\
					&\le E_i\Big[ 1 + \beta_\ip \hip + E_i[\gamma_\ip] \cdot \Hip^* \hip \Big] \norm{\delta Y_\ip}_\infty 									\\
					&\le E_i\Big[ 1 + M_y \hip \Big] \norm{\delta Y_\ip}_\infty 																							\\
					&\le e^{M_y \hip} \norm{\delta Y_\ip}_\infty ,
			\end{align*}
		since $\beta_\ip \le M_y$ by \MonY and $1+x \le e^x$. This proves the second claim. 
		Finally, because $M_y \le 0$, we have $\E_\ip \le 1 + 0 + E_i[\gamma_\ip] \hip \Hip^*$.
		Denoting $\wh \gamma_i = E_i[\gamma_\ip]$, we use the Cauchy--Schwartz inequality to have
			\begin{align*}
				\Abs{\delta Y_i}^2
					&\le E_i\Big[ \Abs{1 + \wh \gamma_i \cdot\Hip^* \hip}^2 \Big] E_i\big[ \Abs{\delta Y_\ip}^2 \big] 								\\	
					& =  E_i\Big[ 1 + 2 \wh \gamma_i \cdot\Hip^* \hip + \abs{\wh \gamma_i \cdot \Hip^*}^2 \hip^2 \Big] E_i\big[ \Abs{\delta Y_\ip}^2 \big] 								\\	
					& =  \Big( 1 + 0 + \abs{\wh \gamma_i}^2 \frac{\Lambda_i}{\hip} \hip^2 \Big) E_i\big[ \Abs{\delta Y_\ip}^2 \big] 								\\	
					&\le \Big( 1 + L_z^2 \hip \Big) E_i\big[ \Abs{\delta Y_\ip}^2 \big] 								\\	
					&\le e^{ L_z^2 \hip } E_i\big[ \Abs{\delta Y_\ip}^2 \big] .
			\end{align*}
		Here, we have used again $E_i[\Hip]=0$, $1+x \le e^x$ and we have used $\Lambda_i \le 1$.		
	\end{proof}

	\begin{corollary}[Comparison property]
		Let $k=1$.
		Consider bounded random variables $\wh Y_\ip$ and $Y_\ip$ such that $\wh Y_\ip \ge Y_\ip$, 
		and consider a second driver $\wh f$ such that $\wh f \ge f$ over $\R \times \R^{1 \times d}$.
		Let $(Y_i, Z_i)$ and $(\wh Y_i, \wh Z_i)$ be the output of one step of the scheme \TimeDiscScheme with respective input-drivers $(Y_\ip,f)$ and $(\wh Y_\ip, \wh f)$.
		If $\hip \le h^0_\ip(\norm{\wh Y_\ip}_\infty,\norm{Y_\ip}_\infty)$ as defined in proposition \ref{proposition---stability.estimate.one-step.via.linearization.1d}, 
		then we have $\wh Y_i \ge Y_i$ ---order is preserved.
	\end{corollary}

This property is known as comparison theorem for the continuous-time BSDE, and its preservation by the discrete-time dynamics is another notion of numerical stability. 
See \cite{ChassagneuxRichou2016,CheriditoStadje2012,CheriditoStadje2013} for the comparison property for implicit schemes. 

	\begin{proof}
		When $\wh f = f$, the result follows immediately from the representation of $\delta Y_\ip$ given by proposition \ref{proposition---stability.estimate.one-step.via.linearization.1d}.		
		In the general case, one only needs to adapt the proof of proposition \ref{proposition---stability.estimate.one-step.via.linearization.1d}. 
		With the previous notation and $\delta f_\ip := \big(\wh f -f \big)(\wh{Y}_\ip,\wh{Z}_i)$, write 
			\begin{align*}
				\wh f(\wh{Y}_\ip,\wh{Z}_i) - f(Y_\ip,Z_i) 
					&= \wh f(\wh{Y}_\ip,\wh{Z}_i) - f(\wh{Y}_\ip,\wh{Z}_i) + f(\wh{Y}_\ip,\wh{Z}_i) - f(Y_\ip,Z_i)		\\
					&= \delta f_\ip + \beta_\ip \delta Y_\ip + \gamma_\ip \cdot \delta Z_i .
			\end{align*}
		This leads to the representation $\delta Y_i = E_i\Big[ \delta Y_\ip \E_\ip + \delta f_\ip \hip \Big]$.
		Seeing as $\E_\ip > 0$, and we have $\delta Y_\ip \ge 0$ and $\delta f_\ip \ge 0$ by assumption, we have $\delta Y_i \ge 0$ indeed.		
	\end{proof}

	\begin{remark}
		The partition $\pi^n$ defined in section \ref{section---preliminaries} does not guarantee that the scheme possesses the comparison property. 
		For this, given terminal conditions $\xi^N$ and $\wh \xi^N$ that are bounded by some common $C$, define $\pi^{n,+}$ like $\pi^n$ but with
			\begin{align}		\label{equation---reference--definition.of.hnip.1dim.with.comparison}
				h^{n,+}_\ip := \min \left\{ \frac{1}{3L_y \big( 1 + 2 \ e^{(m-1) M_y (T-t^n_\ip)} C^{m-1} \big)} \ , \ h_0(L_z) \ , \ \hmax \ , \ \tnip \ , \ \frac{T}{n} \right\} .
			\end{align}
		The factor $2$ here makes each time-step smaller and thus ensures that the size bounds given by proposition \ref{proposition---size.estimate.all.steps.via.linearization.1d} hold.
		Consequently, the condition for comparison indeed holds at each time-step.
	\end{remark}

\subsection{The multi-dimensional case}

In this subsection we treat the general case of a dimension $k \in \N^*$ and we obtain size estimates for $(Y_i,Z_i)$, not just $Y_i$.
First, we analyse one time-step of size $\hip$.

	\begin{proposition}		\label{proposition---size.estimate.one.step.via.squaring.multidim}
		Let $Y_\ip$ be a bounded input for the scheme \TimeDiscScheme.
		Define $\wh h_0	 := \frac{1}{8} \ \frac{1}{2(d+1) L_z^2 }$ 
		and 
			\begin{align*}
				\wh{M}_y(\hip) = M_y + 2 L_z^2 + 2 (d+1) L_y^2 \big( 1 + \Norm{Y_\ip}^{2(m-1)}_\infty \big) \hip .
			\end{align*}
		If $\hip \le \min\{ \wh h_0 , \hmax \}$, then the output $(Y_i,Z_i)$ of the scheme satisfies
			\begin{align}		\label{equation---size.estimate.one.step.almost.sure}
				\abs{Y_i}^2 + \frac{1}{8} \abs{Z_i}^2 \hip \le e^{2 \wh{M}_y(\hip)	 \hip} E_i\big[ \abs{Y_\ip}^2 \big] .
			\end{align}
		Furthermore, if \kDS holds and 
			\begin{align*}
				\hip \le h^0_\ip(\Norm{Y_\ip}_\infty) 
					:= \min \left\{ \frac{1}{4} \ \frac{-M_y}{(d+1) L_y^2 (1 + \Norm{Y_\ip}^{2(m-1)}_\infty) } \ ,\ \wh h_0 \ , \ \hmax \right\} ,
			\end{align*}
		then $\wh{M}_y(\hip) \le M_y/4 =: \wh{M}_y$, which is strictly negative. 
	\end{proposition}

	\begin{proof}
		Similarly to \cite{LionnetDosReisSzpruch2016}, we rewrite equation \eqref{equation---reference--ATS.scheme} as
			\begin{align*}
				Y_i + \Delta M_\ip = Y_\ip + f(Y_\ip,Z_i) \hip ,
			\end{align*}
		where $\Delta M_\ip = Y_\ip + f(Y_\ip,Z_i) \hip - E_i[Y_\ip + f(Y_\ip,Z_i) \hip]$.
		We square both sides of the equation above, then take the conditional expectation, 
		using $E_i\big[2\scalar{Y_i}{\Delta M_\ip}\big] = 2\scalar{Y_i}{E_i[\Delta M_\ip]} = 0$ seeing as $E[\Delta M_\ip]=0$.
		This gives
			\begin{align*}
				\abs{Y_i} + 0 + E_i[\abs{\Delta M_\ip}^2] = E_i\Big[ \abs{Y_\ip}^2 + 2 \scalar{Y_\ip}{f(Y_\ip,Z_i) \hip} + \abs{f(Y_\ip,Z_i)}^2 \hip^2 \Big] .
			\end{align*}
		We assume that $L_z >0$, as the proof when $L_z$ is simpler (see remark \ref{remark---Lz.null.in.size.estimate.one.step.via.squaring.multidim}).  
		Using \MonY, \LipZ, $f(0,0)=0$ and a Young inequality with some $\alpha_z > 0$ yet to specify, we have 
			\begin{align*}
				2 \scalar{Y_\ip}{f(Y_\ip,Z_i)}
					& =  2 \scalar{Y_\ip - 0}{f(Y_\ip,Z_i)-f(0,Z_i)} + 2 \scalar{Y_\ip}{f(0,Z_i) - f(0,0)} 		\\
					&\le 2 M_y \abs{Y_\ip}^2 + 2 \abs{Y_\ip} L_z \abs{Z_i}									\\
					&\le \big( 2 M_y + \alpha_z \big) \abs{Y_\ip}^2 + \frac{L_z^2}{\alpha_z} \abs{Z_i}^2 .
			\end{align*}
		Thus, 
			\begin{align*}
				\abs{Y_i} + E_i[ &\abs{\Delta M_\ip}^2] 																																					\\
						&\le  E_i\bigg[ \Big\{ 1 + \big(2 M_y + \alpha_z\big) \hip \Big\} \abs{Y_\ip}^2 \bigg]
										+ \frac{L_z^2}{\alpha_z} \abs{Z_i}^2 \hip + E_i \big[ \abs{f(Y_\ip,Z_i)}^2 \big] \hip^2 .
			\end{align*}
		
		We now work on the term $\Delta M_\ip$, which can be uniquely written as $\Delta M_\ip = \zeta_i \Lambda_i^{-1} \Hip \hip + \Delta N_\ip$, 
		for a $\F_i$-measurable $\R^{k \times d}$-valued $\zeta_i$ and  a martingale increment $\Delta N_\ip$ orthogonal to $\Hip$ ($E[\Delta N_\ip]=0$ et $E[\Delta N_\ip \, H_\ip^* ] = 0$).
		The random variable $\zeta_i$ is given by $\zeta_i = E_i[(Y_\ip + f(Y_\ip,Z_i) \hip) \Hip^*]$ and we denote by
		$D_i = Z_i - \zeta_i$.
		Using the orthogonality of $\Hip$ and $\Delta N_\ip$, as well as $\zeta_i = Z_i - D_i$ and a Young inequality, we have
			\begin{align*}
				E_i\big[\abs{\Delta M_\ip}^2\big] 
					& = E_i\big[\abs{\zeta_i \Lambda_i^{-1} \Hip \hip + \Delta N_\ip}^2\big] 
							= E_i\big[\abs{\zeta_i \Lambda_i^{-1} \Hip \hip}^2\big] + E_i\big[\abs{\Delta N_\ip}^2\big]												\\
					& = \abs{\zeta_i}^2 \Lambda_i^{-1} \hip + E_i\big[\abs{\Delta N_\ip}^2\big]																				\\
					&\ge \half \abs{Z_i}^2 \Lambda_i^{-1} \hip - \abs{D_i}^2 \Lambda_i^{-1} \hip + E_i\big[\abs{\Delta N_\ip}^2\big] .
			\end{align*}
		Consequently, 
			\begin{align*}
				\abs{Y_i} + \bigg( &\frac{\Lambda^{-1}}{2} - \frac{L_z^2}{\alpha_z} \bigg) \abs{Z_i}^2 \hip  + E_i[\abs{\Delta N_\ip}^2] 		\\
					&\le  E_i\bigg[ \Big\{ 1 + \big(2 M_y + \alpha_z\big) \hip \Big\} \abs{Y_\ip}^2 \bigg] + \abs{D_i}^2 \Lambda_i^{-1} \hip
						 + E_i \big[ \abs{f(Y_\ip,Z_i)}^2 \big] \hip^2 .
			\end{align*}

		We now estimate $\abs{D_i} = \abs{Z_i - \zeta_i}$. 
		Using the definitions of $Z_i$ and $\zeta_i$, we have
			\begin{align*}
				D_i = E_i\Big[ -f(Y_\ip,Z_i) \hip \Hip^* \Big] .
			\end{align*}		
		So, using the Cauchy--Schwartz inequality, 
			\begin{align*}
				\abs{D_i}^2 \Lambda_i^{-1} \hip 
					&\le \Lambda_i^{-1} \hip E_i\Big[ \abs{ \Hip^* \hip}^2 \Big] E_i\Big[ \abs{ f(Y_\ip,Z_i) }^2  \Big] 			\\
					&\le d \hip^2 E_i\Big[ \abs{ f(Y_\ip,Z_i) }^2  \Big] .
			\end{align*}
		Meanwhile, using \RegY and \LipZ, 
			\begin{align*}
				E_i \big[ \abs{f(Y_\ip,Z_i)}^2 \big] \hip^2 \le 4 L_y^2 E_i \big[ ( 1 + \abs{Y_\ip}^{2(m-1)} ) \abs{Y_\ip}^{2} \big] \hip^2 + 2 L_z^2 \abs{Z_i}^2 \hip^2 .
			\end{align*}
		Therefore, using also $\Lambda^{-1} \ge 1$, we have the estimate
			\begin{align*}
				\abs{Y_i} + \bigg( \half &- \frac{L_z^2}{\alpha_z} \bigg) \abs{Z_i}^2 \hip  + E_i[\abs{\Delta N_\ip}^2] 		\\
					&\le  E_i\bigg[ \Big\{ 1 + \big(2 M_y + \alpha_z\big) \hip \Big\} \abs{Y_\ip}^2 \bigg] + 2(d+1) L_z^2 \abs{Z_i}^2 \hip^2 					\\
						&\hspace{2cm} + 4 (d+1) L_y^2 E_i \big[ ( 1 + \abs{Y_\ip}^{2(m-1)} ) \abs{Y_\ip}^{2} \big] \hip^2 .
			\end{align*}
					
		Now, we choose $\alpha_z = 4 L_z^2$, 
		so that $\half - \frac{L_z^2}{\alpha_z} = \frac{1}{4}$.
		We have taken $\hip \le \wh h_0$ 
		so it is small enough that $2(d+1) L_z^2 \hip \le \frac{1}{8}$.
		With this, we have	
			\begin{align*}
				\abs{Y_i} + \frac{1}{8} \abs{Z_i}^2 \hip
					&\le E_i\bigg[ \Big\{ 1 + 2 \wh{M}_y(\hip) \hip \Big\} \abs{Y_\ip}^2 \bigg] 					
					\le E_i\bigg[ e^{2 \wh{M}_y(\hip) \hip} \abs{Y_\ip}^2 \bigg] ,
			\end{align*}
		since $1+x \le e^x$ for $x \in \R$ and we have set
			\begin{align*}
				\wh{M}_y(\hip) = M_y + 2 L_z^2 + 2 (d+1) L_y^2 \big( 1 + \Norm{Y_\ip}^{2(m-1)}_\infty \big) \hip .
			\end{align*}			
		
		To conclude, note that using \kDS (i.e. $\frac{L_z^2 }{-M_y} \le \frac{1}{4}$) 
		and the value of $h^0_\ip(\Norm{Y_\ip}_\infty)$, 
		the second and third term of $\wh M_y\big(h^0_\ip(\Norm{Y_\ip}_\infty)\big)$ are bounded above by $\frac12 (-M_y)$ and $\frac14 (-M_y)$ respectively. 
		So $\wh M_y(\hip) \le \frac14 M_y < 0$ indeed.
	\end{proof}

	\begin{remark}		\label{remark---Lz.null.in.size.estimate.one.step.via.squaring.multidim}
		In the case $L_z = 0$, that is to say $f$ does not depend on $z$, the above proof has to be modified into a simpler one : take $\alpha_z = 0$ and use the convention $\frac{L_z^2}{\alpha_z}=0$.
		Also, take $\wh h_0 = +\infty$. Finally, $2 L_y^2$ becomes $L_y^2$.
	\end{remark}

	\begin{remark}
		If instead of \TimeDiscScheme we considered the same scheme but with $Z_i$ defined as $Z_i := E_i\Big[ \big( Y_\ip + f(Y_\ip,0) \hip  \big) \Hip \Big]$, 
		the term $2 (d+1) L_y^2$ can be replaced by $2 L_y^2$ in $\wh M_y$ and the term $2 (d+1) L_z^2$ by $(d + 2) L_z^2$ in $\wh h_0$.
		That makes $\wh M_y(\hip)$ smaller and $h^0_\ip(\Norm{Y}_\infty)$ bigger : the scheme is numerically stable from a bigger size of the time steps.
		So it is in a sense more numerically stable.
	\end{remark}

	\begin{remark}
		Notice that the condition \kDS is not sharp. 

		Firstly, in order to obtain $\wh M_y < 0$, we imposed that the second and third term be bounded above by $-\half M_y$ and $-\frac14 M_y$. 
		Clearly, one could impose a smaller upper bound for the third term (leading to requiring a smaller $h^0_\ip(\Norm{Y_\ip}_\infty)$) and have more room for the second term.
		Specifically, we could replace the condition $\half \alpha_z \le -\half M_y$ by $\half \alpha_z \le - \mu M_y$, for $\mu \in [\half,1[$.
		
		Secondly, we could define $\alpha_z = \frac{2}{\nu} L_z^2$, for $\nu \in [\half,1[$ instead of $\nu=\half$, 
		leading to a term $\half - \frac{L_z^2}{\alpha_z} = \frac{1-\nu}{2}$ instead of $\frac14$. 
		We would then define a smaller $\wh h_0 = \frac{1-\nu}{4} \frac{1}{2(d+1) L_z^2}$,
		and we would end up with $\frac{1-\nu}{4}$ on the LHS instead of $\frac18$. 
		This smaller $\alpha_z$ leads to a smaller second term in $\wh M_y$. 
		
		Combining the two, the condition \kDS could be replaced by $\frac{L_z^2 }{-M_y} \le  \mu \nu$. So the sharper condition is rather $\frac{L_z^2 }{-M_y} < 1$. 
		The closer to $1$ this ratio is, the smaller the time-steps would have to be.
		What is important here is not so much the sharpness of the numerical constants in \kDS, as the fact that $f$ must be so that $\frac{L_z^2}{-M_y} \le c$.
	\end{remark}

	\begin{remark}
		When $M_y=0$ (and then necessarily $L_z=0$), since we have $0 = M_y < \wh M_y(\hip)$ for all $\hip > 0$, 
		we see that the analysis above does not guarantee the numerical stability of the scheme.
	\end{remark}

We now consider, for a fixed $n \in \N^*$, the grid $\pi^n$ and the associated approximations $(Y^n_i,Z^n_i)_{i = 0 \ldots N}$.
Recall that $\wh M_y = M_y/4 \le 0$.

	\begin{corollary}		\label{proposition---size.estimate.all.steps.via.squaring.multidim}
		Let \kdSn holds. We have, for all $i$,
			\begin{align*}
				\Norm{Y^n_i}_\infty \le e^{\wh{M}_y (T-\tni)} \Norm{\xi^N}_\infty \le \Norm{\xi^N}_\infty.
			\end{align*}
	\end{corollary}

	\begin{proof} 
		The result follows from a simple backward induction and the definition of $\pi^n$.
	\end{proof}

\section{Global error of the scheme}			\label{section---convergence.of.the.scheme}

In this section we analyze the global error of the scheme, compared to the continuous-time solution. 
Note that the convergence is not a real issue, since the terminal condition is bounded and the partition $\pi^n$ becomes the regular one asymptotically. But the goal here is to determine non-asymptotic error bounds that are as small as possible and valid for as large a $\abs{\pi}$ as possible.

\subsection{Stability of the scheme}
\label{section---convergence--stability.of.the.scheme}

For an arbitrary time-grid $\pi$, we consider bounded inputs $\wh{Y}_\ip$ and $Y_\ip$ at time $\tip$, and the respective outputs $(\wh{Y}_i,\wh{Z}_i)$ and $(Y_i,Z_i)$ at time $\ti$ from one step of the scheme \TimeDiscScheme. We denote generically by $\delta x$ the quantity $\wh{x}-x$.

	\begin{proposition}		\label{proposition---stability.of.the.scheme}
		Let $\wh Y_\ip$ and $Y_\ip$ be bounded inputs for \TimeDiscScheme, and let $(\wh{Y}_i,\wh{Z}_i)$ and $(Y_i,Z_i)$ be the respective outputs.
		Assume that $\hip \le \wh h_0$, where we recall that $\wh h_0 = \frac{1}{8} \ \frac{1}{2(d+1) L_z^2}$,
		and define $\wt M_{y,\ip} = \wt M_y(\hip,\norm{\wh{Y}_\ip}^{2(m-1)}_\infty,\norm{Y_\ip}^{2(m-1)}_\infty)$ as
			\begin{align*}
				\wt M_{y,\ip} = M_y + 2 L_z^2 + 3 (d+1) L_y^2 \hip \big( 1 + \norm{\wh{Y}_\ip}^{2(m-1)}_\infty + \norm{Y_\ip}^{2(m-1)}_\infty \big) .
			\end{align*}		
		Then,
			\begin{align*}
				\abs{\delta Y_\ip}^2 + \frac{1}{8} \abs{\delta Z_i}^2 \hip 
						\le e^{ 2 \wt M_{y,\ip} \hip } E_i\big[ \abs{\delta Y_\ip}^2 \big] .
			\end{align*}				
	\end{proposition}

	\begin{proof}
		Again, we write $\delta Y_i + \delta \Delta M_\ip = \delta Y_\ip + \delta f_\ip \hip$, where $\delta \Delta M_\ip = \delta Y_\ip + \delta f_\ip \hip - E_i[\delta Y_\ip + \delta f_\ip \hip]$ 
		and $\delta f_\ip = f(\wh{Y}_\ip,\wh{Z}_i) - f(Y_\ip,Z_i)$.
		Squaring and taking conditional expectation, we have
			\begin{align*}
				\abs{\delta Y_\ip}^2 + 0 + E_i[\abs{\delta \Delta M_\ip}^2] = E_i\Big[ \abs{\delta Y_\ip}^2 + 2 \scalar{\delta Y_\ip}{\delta f_\ip \hip} + \abs{\delta f_\ip}^2 \hip^2 \Big] . 
			\end{align*}
		Using \MonY and \LipZ we have, for some $\alpha_z >0$ yet to be chosen, 
			\begin{align*}
				\scalar{\delta Y_\ip}{\delta f_\ip}
					& =  \scalar{\delta Y_\ip}{ f(\wh{Y}_\ip,\wh{Z}_i) - f(Y_\ip,\wh{Z}_i) } + \scalar{\delta Y_\ip}{ f(Y_\ip,\wh{Z}_i) - f(Y_\ip,Z_i) }										\\
					&\le M_y \abs{\delta Y_\ip}^2 + \abs{\delta Y_\ip} L_z \abs{\delta Z_i}																													\\
					&\le \left( M_y + \frac{\alpha_z}{2} \right) \abs{\delta Y_i}^2 + \frac{L_z^2}{2\alpha_z} \abs{\delta Z_i}^2 .
			\end{align*}
		So we have
			\begin{align*}
				\abs{\delta Y_\ip}^2 + E_i[\abs{\delta \Delta M_\ip}^2] 
					\le E_i\Big[ \Big\{ 1 + \big[ 2 M_y + \alpha_z \big] \hip \Big\} \abs{\delta Y_\ip}^2 + \abs{\delta f_\ip}^2 \hip^2 \Big] + \frac{L_z^2}{\alpha_z} \abs{\delta Z_i}^2 \hip . 
			\end{align*}
		
		Now we handle the term $\delta \Delta M_\ip$.
		It can be uniquely decomposed as $\delta \Delta M_\ip = \delta \zeta_i \Lambda_i^{-1} \Hip \hip + \delta \Delta N_\ip$, where $\delta \Delta N_\ip$ is a martingale increment orthogonal to 
		$\Hip$ and $\delta \zeta_i = E_i\big[  \big( \delta Y_\ip + \delta f_\ip \hip \big)  \Hip^* \big]$.
		By orthogonality and using a Young inequality,
			\begin{align*}
				E_i[\abs{\delta \Delta M_\ip}^2] 
					& = E_i[\abs{\delta \zeta_i \Lambda_i^{-1} \Hip \hip}^2 ] + E_i[\abs{\delta \Delta N_\ip}^2]																												\\
					& = \abs{\delta \zeta_i}^2 \Lambda_i^{-1} \hip  + E_i[\abs{\delta \Delta N_\ip}^2]																																			\\
					&\ge \half \abs{\delta Z_i}^2 \Lambda_i^{-1} \hip - \abs{\delta D_i}^2 \Lambda_i^{-1} \hip + E_i[\abs{\delta \Delta N_\ip}^2] ,
			\end{align*}
		since $\delta \zeta_i = \delta Z_i - \delta D_i$. We therefore have
			\begin{align*}
				\abs{\delta Y_\ip}^2 +& \left( \frac{\Lambda_i^{-1}}{2} - \frac{L_z^2}{\alpha_z} \right) \abs{\delta Z_i}^2 \hip + E_i[\abs{\delta \Delta N_\ip}^2]									\\
					&\le E_i\Big[ \Big\{ 1 + \big[ 2 M_y + \alpha_z \big] \hip \Big\} \abs{\delta Y_\ip}^2 + \abs{\delta f_\ip}^2 \hip^2 \Big]  +  \abs{\delta D_i}^2 \Lambda_i^{-1} \hip . 
			\end{align*}
		
		We now estimate $\abs{\delta D_i}^2 \Lambda_i^{-1} \hip$. We have
			\begin{align*}
				\delta D_i = \delta Z_i - \delta \zeta_i = E_i\left[ \bigg\{ 0 - \Big( f(\wh{Y}_\ip,\wh{Z}_i)-f(Y_\ip,Z_i) \Big) \hip  \bigg\} \Hip^* \right] 
			\end{align*}		
		so, using the Cauchy--Schwartz inequality, 
			\begin{align*}
				\abs{\delta D_i}^2 \Lambda_i^{-1} \hip 
					&\le \Lambda_i^{-1} \hip E_i\left[ \abs{\Hip \hip}^2 \right] E_i\left[ \Abs{ f(\wh{Y}_\ip,\wh{Z}_i)-f(Y_\ip,Z_i) }^2 \right]																					\\
					&\le d \hip^2 E_i\left[ \abs{\delta f_\ip}^2 \right] .
			\end{align*}
		
		Using \RegY and \LipZ, we then estimate
			\begin{align*}
				\abs{\delta f_\ip}^2
					&\le 6 L_y^2 \big( 1 + \abs{\wh{Y}_\ip}^{2(m-1)} + \abs{Y_\ip}^{2(m-1)} \big) \abs{\delta Y_\ip}^2 + 2 L_z^2 \abs{\delta Z_i}^2 .
			\end{align*}
		Gathering these estimates, we have
			\begin{align*}
				\abs{\delta Y_\ip}^2 +& \left( \half - \frac{L_z^2 \Lambda_i}{\alpha_z} \right) \abs{\delta Z_i}^2  \Lambda_i^{-1} \hip + E_i[\abs{\delta \Delta N_\ip}^2]										\\
					&\le E_i\Big[ \big\{ 1 + [ 2 M_y + \alpha_z ] \hip \big\} \abs{\delta Y_\ip}^2 \Big]
								+ 2 (d+1) L_z^2 \abs{\delta Z_i}^2 \hip^2 																																\\
							&\qquad \qquad +  6 (d+1) L_y^2 \hip^2 E_i\Big[ \big( 1 + \abs{\wh{Y}_\ip}^{2(m-1)} + \abs{Y_\ip}^{2(m-1)} \big) \abs{\delta Y_\ip}^2 \Big] ,
			\end{align*}				
		which then implies, since $\Lambda_i^{-1} \ge 1$, that
			\begin{align*}
				&\abs{\delta Y_\ip}^2 
							+ \left( \half - \frac{L_z^2}{\alpha_z} - 2 (d+1) L_z^2 \hip \right) \abs{\delta Z_i}^2 \hip 																\\
					&\le E_i\bigg[ \Big\{ 1 + \big[ 2 M_y + \alpha_z 
															+ 6 (d+1) L_y^2 \hip ( 1 + \norm{\wh{Y}_\ip}^{2(m-1)}_\infty + \norm{Y_\ip}^{2(m-1)}_\infty ) \big] \hip \Big\} \abs{\delta Y_\ip}^2 \bigg] .
			\end{align*}				
		We therefore take $\alpha_z$ such that $\frac{L_z^2}{\alpha_z} = \frac14$, i.e. $\alpha_z = 4  L_z^2$.
		Also, we have assumed that $\hip \le \wh h_0 = \frac18 \, \frac{1}{2 (d+1) L_z^2}$ so $\le 2 (d+1) L_z^2 \hip \le \frac18$. 
		Thus we have indeed
			\begin{align*}
				\abs{\delta Y_\ip}^2 + \frac{1}{8} \abs{\delta Z_i}^2  \hip 
						\le e^{ 2 \wt M_y(\hip,\norm{\wh Y_\ip}_\infty,\norm{Y_\ip}_\infty) \, \hip } E_i\Big[ \abs{\delta Y_\ip}^2 \Big] ,
			\end{align*}				
		as desired.
 	\end{proof}

We will want to use this $L^2$-stability estimate for the inputs $Y_\tip$ and $Y_\ip$ (the BSDE solution and its approximation, at time $\tip$).
In doing this, we may want to ensure that $\wt M_y(\hip,\norm{Y_\tip}_\infty,\norm{Y_\ip}_\infty) < 0$. 
This is not automatically guaranteed (non-asymptotically, that is) when using the time-grid $\pi^n$, but can easily be obtained by using a grid with smaller time-steps.
Therefore, we introduce the grid $\wt \pi^n = (\wt t^n_i)_{i= 0 \ldots \wt N}$ constructed like $\pi^n$ but with smaller time-steps $\wt h^n_\ip$ defined by
	\begin{align*}
		\wt h^n_\ip = \min\left\{ \frac14 \ \frac{-M_y}{3 (d+1) L_y^2 \big( 1 + 2 \ e^{2(m-1) \wh M_y (T-\wt t^n_\ip)} \Norm{g}_\infty^{2(m-1)} \big)} \ , \ \hnip(\wt t^n_\ip) \right\} ,
	\end{align*}
where $\hnip(\wt t^n_\ip)$ is given by \eqref{equation---reference--definition.of.hnip.1dim}-\eqref{equation---reference--definition.of.hnip.multidim} but with $\tnip$ replaced by $\wt t^n_\ip$.
We still denote by $(Y^n_i,Z^n_i)_{i = 0 \ldots \wt N}$ the resulting approximations.

Because the time-steps $\wt h^n_\ip$ are smaller than $\hnip$, we can obtain again from 
corollary \ref{proposition---size.estimate.all.steps.via.linearization.1d} or \ref{proposition---size.estimate.all.steps.via.squaring.multidim} that 
$\Norm{Y^n_i}_\infty \le e^{\wh M_y (T-\wt t^n_i)} \Norm{g}_\infty$. 
Meanwhile, for the BSDE solution, we have $\norm{Y_{\wt t^n_i}}_\infty \le e^{M_y (T-\wt t^n_i)} \Norm{g}_\infty \le e^{\wh M_y (T-\wt t^n_i)} \Norm{g}_\infty$. 
So we have  
	\begin{align*}
		\wt M_{y}(& \hip, \norm{Y_{\wt t^n_\ip}}_\infty, \norm{Y^n_\ip}_\infty) 																																			\\
			&\le \wt M_{y}^{g}(\wt t^n_\ip,\wt h^n_\ip) := M_y + 2 L_z^2 + 3 (d+1) L_y^2 \ \wthnip \big( 1 + 2 \ e^{2(m-1) \wh M_y (T-\wttnip)} \Norm{g}_\infty^{2(m-1)} \big) .
	\end{align*}
Now, by \kdSn we have $2L_z^2 \le (-M_y)/2$ and
by definition of $\wt h^n_\ip$ we have the third term bounded above by $(-M_y)/4$.
That means we have $ M_{y}^{g}(\wttnip,\wthnip) \le \frac14 M_y = \wh M_y$.

\paragraph*{} 
Regardless of whether the time-grid $\pi^n$ or $\wt \pi^n$ is used, we certainly have $\norm{Y_\tip}_\infty$ and $\norm{Y^n_\ip}_\infty \le \Norm{g}_\infty$. 
So we will always have $\wt M_{y}(\hip, \norm{Y_{\wt t^n_\ip}}_\infty, \norm{Y^n_\ip}_\infty) \le \wt M_{y}^{g}(\hip)$ where 
	\begin{align}		\label{equation---definition.of.wtMg.of.h}
		\wt M_{y}^{g}(\hip) := M_y + 2 L_z^2 + 3 (d+1) L_y^2 \ \hip \big( 1 + 2 \Norm{g}_\infty^{2(m-1)} \big) .
	\end{align}
This $\wt M_{y}^{g}(\hip)$ is not guaranteed to be negative, but will become so as $\hip \goesto 0$.

\subsection{Local errors due to the time-discretization}		
\label{section---convergence--time-discretization.error}

In this subsection, given a partition $\pi$, we estimate the total error due to the time-discretization, defined as the sum of the local time-discretization errors.
For all $i = 0 \ldots N-1$, we denote from now on by $(\wh Y_i,\wh Z_i)$ the output of one step of the scheme \TimeDiscScheme when the input is $Y_\tip$. 
So here
	\begin{align*}
		\left\{\begin{aligned}
			\wh Y_i &= E_i\Big[ Y_\tip + f(Y_\tip,Z_i) \hip \Big]											\\
			\wh Z_i &= E_i\Big[ Y_\tip \Hip^* \Big] .
		\end{aligned}\right.
	\end{align*}

	\begin{proposition}		
	\label{proposition---estimation.of.the.discretization.error}	
		There exists constants $C^Z, C^Y \ge 0$ (independent of $\pi$) such that 
			\begin{align*}
				\sum_{i=0}^{N-1} \bE\big[\, \abs{\overline{Z}_\ti - \widehat{Z}_i}^2 \big] h \le C^Z \ \abs{\pi}
				\qquad\text{and}\qquad 
				\sum_{i=0}^{N-1} \frac{\bE\big[\, \abs{Y_\ti - \widehat{Y}_i}^2 \big]}{\hip} \le C^Y \ \abs{\pi}  .
			\end{align*}
	\end{proposition}

This proof is very similar to the proof of the corresponding result in \cite{LionnetDosReisSzpruch2016}. 
It is given in the appendix for completeness.

\subsection{Global error estimate}

We now estimate the total error between the numerical approximation and the solution to the continuous-time BSDE.
More specifically, we want to estimate the following errors:
	\begin{align*}
		\mathrm{Error}_Y(n)^2 = \max_{i = 0 \ldots N^n} E\big[ \abs{Y_\tni - Y^n_i}^2 \big]   
		\quad \text{and} \quad 
		\mathrm{Error}_Z(n)^2 = E\bigg[ \sum_{i=0}^{N^n} \Abs{\wb Z_\tni - Z^n_i}^2 \hnip \bigg] .
	\end{align*}
We recall that the notation $\wt M_y^g(h)$ was introduced in \eqref{equation---definition.of.wtMg.of.h}.

	\begin{theorem}			\label{theorem---global.error.estimate}
		For every $\epsilon > 0$, there exists a constant $C_\epsilon \ge 0$ such that, for all $n$, 
			\begin{align*}
			 	\mathrm{Error}_Y(n)^2 \quad \text{and} \quad \mathrm{Error}_Z(n)^2 
			 		\le e^{(\epsilon+\wt M_y^g(\abs{\pi^n})) T} \ C_\epsilon \, \abs{\pi^n} \le e^{(\epsilon+\wt M_y^g(\hmax)) T} \ C_\epsilon \, \frac{T}{n} .
			\end{align*}
	\end{theorem}

	\begin{proof}
		The proof consists, classically, in seeing the error at time $i$ as resulting from the error existing at time $i+1$ and an additional, local error. 
		The stability of the scheme (analyzed in subsection \ref{section---convergence--stability.of.the.scheme}) allows to control how existing errors propagate backward in time, 
		while the analysis of the total time-discretization error (done in subsection \ref{section---convergence--time-discretization.error}) allows to control the sum of local errors.
		
		Specifically, recalling the definition of $(\wh Y_i,\wh Z_i)$ from subsection \ref{section---convergence--time-discretization.error}, we write 
			\begin{align*}
				Y_\ti - Y_i = \big( Y_\ti - \wh Y_i \big) + \big( \wh Y_i - Y_i)
				\qquad \text{and} \qquad
				\wb Z_\ti - Z_i = \big( \wb Z_\ti - \wh Z_i \big) + \big( \wh Z_i - Z_i).
			\end{align*}
		Then, using this decomposition and a Young inequality, we have for any $\epsilon > 0$
			\begin{align*}
				E[\abs{Y_\ti-Y_i}^2] &\le \Big(1+\frac{1}{\epsilon \hnip} \Big) E[\abs{Y_\ti - \wh Y_i}^2] + (1+\epsilon \hnip) E[\abs{\wh Y_i - Y_i}^2] , \qquad \text{and}						\\
				E[\abs{\wb Z_\tni - Z_i}^2]  &\le 2 E[\abs{\wb Z_\ti - \wh Z_i}^2] + 2 E[\abs{\wh Z_i - Z_i}^2] .
			\end{align*}
		We will denote by $\tau^Y_i = E[\abs{Y_\ti - \wh Y_i}^2]$ and $\tau^Z_i = E[\abs{\wb Z_\ti - \wh Z_i}^2] \hip$, the local time-discretization errors on $Y$ and $Z$.
		Therefore, since $\hnip \le T$ and $1 \le 1+\epsilon \hnip \le e^{\epsilon \hnip}$, and using the $L^2$-stability of proposition \ref{proposition---stability.of.the.scheme}, 
			\begin{align*}
				E[ & \abs{Y_\ti-Y_i}^2] + \frac{1}{16} E[\abs{\wb Z_\tni - Z_i}^2] \hip 																															\\
					&\le (1+\epsilon \hnip) \Big( E[\abs{\wh Y_i - Y_i}^2] + \frac18 E[\abs{\wh Z_i - Z_i}^2] \hip \Big) + 
							\Big(\frac{T}{\hnip}+\frac{1}{\epsilon \hnip} \Big) \tau^Y_i + \frac18 \tau^Z_i 																												\\
					&\le e^{(\epsilon+\wt M_y^g(\hnip)) \hnip} E[\abs{Y_\tnip - Y_\ip}^2] + (T+\epsilon^{-1}) \frac{\tau^Y_i}{\hip} + \tau^Z_i .
			\end{align*}
		Then, ``iterating'' this estimate (see lemma A.3 in \cite{LionnetDosReisSzpruch2016}), we obtain for all $i$ that
			\begin{align*}
				E[ \abs{Y_\ti-Y_i}^2] + & \frac{1}{16} E\bigg[ \sum_{j=i}^{N-1} \abs{\wb Z_{t^n_j} - Z_j}^2 \hip \bigg]  																			\\
					&\le \exp\bigg(\sum_{j=i}^{N-1} (\epsilon +\wt M_y^g(\hnip)) \hnip\bigg) E[\abs{\xi - \xi^n}^2] 																						\\
							&\qquad + \exp\bigg(\sum_{j=i+1}^{N-1} (\epsilon +\wt M_y^g(\hnip)) \hnip\bigg) \sum_{j=i}^{N-1} (T+\epsilon^{-1}) \frac{\tau^Y_j}{h_{j+1}} + \tau^Z_j .
			\end{align*}
		So we see that we have in the end
			\begin{align*}
				E[ \abs{Y_\ti-Y_i}^2] + \frac{1}{16} E\bigg[ \sum_{j=i}^{N-1} & \abs{\wb Z_{t^n_j} - Z_j}^2 \hip \bigg]																				\\
					&\le e^{(\epsilon +\wt M_y^g(\abs{\pi^n})) (T-\tni)} \bigg[ E[\abs{\xi - \xi^n}^2] + \sum_{i=0}^{N-1} (T+\epsilon^{-1}) \frac{\tau^Y_i}{\hip} + \tau^Z_i  \bigg].
			\end{align*}
		By \eqref{equation---reference--error.bound.on.the.terminal.condition}, $E[\abs{\xi - \xi^n}^2] \le C^X \abs{\pi^n}$, 
		while proposition \ref{proposition---estimation.of.the.discretization.error} guarantees that the last sum is bounded above by 
		$(T+\epsilon^{-1}) C^Y \abs{\pi^n} + C^Z \abs{\pi^n}$.
		Consequently, we only need to set $C_\epsilon = C^X + (T+\epsilon^{-1}) C^Y + C^Z$, so that
			\begin{align*}
			 	\mathrm{Error}_Y(n)^2 \quad \text{and} \quad \mathrm{Error}_Z(n)^2 
			 				\le e^{(\epsilon+\wt M_y^g(\abs{\pi^n})) T} \ C_\epsilon \, \abs{\pi^n} \le e^{(\epsilon+\wt M_y^g(\hmax)) T} \ C_\epsilon \, \frac{T}{n} .
			\end{align*}		
	\end{proof}

	\begin{remark}
		The result above is stated by default for the partition $\pi^n$. 
		However, as discussed after proposition \ref{proposition---stability.of.the.scheme}, the use of the grid $\wt \pi^n$ guarantees that 
		the constant appearing in the exponential term at each step is bounded above by $\wh M_y = M_y/4$. 
		So the global estimate from theorem \ref{theorem---global.error.estimate} would then be
			\begin{align*}
			 	\mathrm{Error}_Y(n)^2 \quad \text{and} \quad \mathrm{Error}_Z(n)^2 
			 				\le e^{(\epsilon+\wh M_y) T} \ C_\epsilon \, \frac{T}{n} .
			\end{align*}		
		For $\epsilon$ sufficiently small (say, smaller than $(-M_y)/8$) this guarantees that the constant $\epsilon+\wh M_y$ is strictly negative, while 
		$\wt M_y^g(\hmax)$ can be big, leading to a huge constant $e^{(\epsilon+\wt M_y^g(\hmax)) T}$.
		We thus obtain with $\wt \pi^n$ a much better non-asymptotic estimate.
		
		Asymptotically however, as $\abs{\pi^n}$ or $\abs{\wt \pi^n} \goesto 0$, it is better to use the more precise error estimate 
		$e^{(\epsilon+\wt M_y^g(\abs{\pi^n})) T} \ C_\epsilon \, \abs{\pi^n}$, since $M_y^g(\abs{\pi^n}) \goesto M_y + 2 L_z^2 \le M_y/2$. 		
	\end{remark}

	\begin{remark}
 		If we are only interested in estimating $\mathrm{Error}_Y(n)^2$, then the proof can be simplified. 
	 	Indeed, the ``iteration'' step really is a simple iteration of the estimate.

		Then, in dimension $1$, instead of the stability estimate from proposition \ref{proposition---stability.of.the.scheme} 
 		one can use the estimate obtained via linearization in proposition \ref{proposition---stability.estimate.one-step.via.linearization.1d} 
 		and end up with the term $\exp\big[ (\epsilon+L_z^2) T \big]$ instead of $\exp\big[(\epsilon+\wt M_y^g(\abs{\pi^n})) T \big]$.
 		This is however not very optimal.
	\end{remark}

\section{Extension to unbounded terminal conditions}		\label{section---extension.to.unbounded.term.cond.}

When the terminal condition $\xi = g(X_T)$ is not bounded, as was assumed so far, one natural solution is to truncate the terminal condition, as was proposed in \cite{LionnetDosReisSzpruch2015}. 
For an arbitrary $L > 0$, let $T^L$ be the projection in the $L^\infty$-norm (i.e. component-wise) in $\R^k$, let $g^L =T^L \circ g$ and $\xi^{n,L} = g^L(X^n_N) = T^L\big( g(X^n_N) \big) = T^L(\xi^n)$.

\paragraph*{}
The ATS truncated scheme 
is defined as follows. 
Take $L_n = L_0 n^{{\alpha}/{(2(m-1))}}$, where $\alpha \in ]0,1]$ and the adapted time-steps grid $\pi^n$ (or $\wt \pi^n$) but computed with $\norm{g}_\infty$ replaced by $\norm{g^{L_n}}_\infty=L_n$.
The scheme produces approximations denoted by $(Y^{n,t}_i,Z^{n,t}_i)_{i = 0 \ldots N}$. We are now interested in the errors 
	\begin{align*}
		\mathrm{Error}_{Y,t}(n)^2 = \max_{i = 0 \ldots N^n} E\big[ \abs{Y_\tni - Y^{n,t}_i}^2 \big]  
		\quad \text{and} \quad 
		\mathrm{Error}_{Z,t}(n)^2 = E\bigg[ \sum_{i=0}^{N^n} \Abs{\wb Z_\tni - Z^{n,t}_i}^2 \hnip \bigg] .
	\end{align*}

	\begin{theorem}
		There exist constants $c$ and $C \ge 0$ such that
			\begin{align*}
			 	\mathrm{Error}_{Y,t}(n)^2 \quad \text{and} \quad \mathrm{Error}_{Z,t}(n)^2 
			 		\le \frac{C}{n} + e^{(\epsilon+\wt M_y^{g^{L_n}}(\abs{\pi^n})) T} \ C_\epsilon \, \abs{\pi^n} 
			 		\le\frac{C}{n} +  e^{(\epsilon+c) T} \ C_\epsilon \, \frac{T}{n} .
			\end{align*}			
	\end{theorem}

	\begin{proof}
		Let us denote by $(Y^t_u,Z^t_u)_{u \in [0,T]}$ the solution to the BSDE \eqref{equation---reference--BSDE} with terminal condition $\xi^{L_n} = g^{L_n}(X_T)$.
		The standard stability estimate for BSDEs with monotone drivers (see \cite{Pardoux1999} or \cite{BriandCarmona2000}) ensures that there exists $C \ge 0$ such that
			\begin{align*}
				\max_{i = 0 \ldots N^n} E\big[ \abs{Y_\tni - Y^{t}_\tni}^2 \big] + E\bigg[ \int_0^T \Abs{Z_u - Z^{t}_u}^2 du \bigg] \le C E\Big[ \Abs{\xi - \xi^{L_n}}^2 \Big] .
			\end{align*}
		On the one hand, we note that
			\begin{align*}
				E\bigg[ \sum_{i=0}^{N^n} \Abs{\wb Z_\tni - \wb Z^t_\tni}^2 \hnip \bigg]
					& =  E\Bigg[ \sum_{i=0}^{N^n} \Abs{ E_i\Big[ \frac{1}{\hnip} \int_\tni^\tnip Z_u du \Big] - E_i\Big[ \frac{1}{\hnip} \int_\tni^\tnip Z^t_u du \Big] }^2 \hnip \Bigg]				\\
					&\le E\Bigg[ \sum_{i=0}^{N^n} E_i\Big[ \hnip \int_\tni^\tnip \abs{Z_u-Z^t_u}^2 du \Big]  \frac{1}{\hnip} \Bigg]																							\\
					& = E\bigg[ \int_0^T \Abs{Z_u - Z^{t}_u}^2 du \bigg] .
			\end{align*}
		On the other hand, using a Markov inequality with $p = \frac{4(m-1)}{\alpha}$, 
			\begin{align*}
				E\Big[ \Abs{\xi - \xi^{L_n}}^2 \Big] 
					&= E\Big[ \Abs{\xi - \xi^{L_n}}^2 1_{\{\abs{\xi}_\infty > L_n\}} \Big]
						\le E\Big[ \abs{\xi}^2 1_{\{\abs{\xi} > L_n\}} \Big]																																														\\
					&\le E\big[ \abs{\xi}^4 \big]^\half P\big( \abs{\xi}^p > (L_n)^p \big)^\half
						\le E\big[ \abs{\xi}^4 \big]^\half E\big[ \abs{\xi}^p \big]^\half (L_n)^{-p/2}																																			\\
					&\le c_4 \ c_p \ (L_0)^{-p/2} \ n^{- \frac{\alpha}{2(m-1)} \frac{p}{2}}
						= C \frac{1}{n} ,
			\end{align*}
		where we used the linear growth of $g$ (stemming from it being Lipschitz) and the bound \eqref{equation---reference--bounds.on.X.Y.Z} for $X_T$ to bound the moments of $\xi=g(X_T)$.
		So we have established that
			\begin{align*}
				\max_{i = 0 \ldots N^n} E\big[ \abs{Y_\tni - Y^{t}_\tni}^2 \big] + E\bigg[ \sum_{i=0}^{N^n} \Abs{\wb Z_\tni - \wb Z^t_\tni}^2 \hnip \bigg] \le \frac{C}{n} .
			\end{align*}
		
		Meanwhile, notice that in the proof of theorem \ref{theorem---global.error.estimate}, we used the fact that $E[\abs{\xi-\xi^n}^2] \le C^X \abs{\pi^n}$.
		When $g$ is replaced by $g^{L_n}$, we have, since $T^{L_n}$ is $1$-Lipschitz,
			\begin{align*}
				E[\abs{\xi^{L_n}-\xi^{n,L_n}}^2] &\le E\Big[ \Abs{T^{L_n}\big(g(X_T)\big) - T^{L_n}\big(g(X^n_N)\big)}^2 \Big]					\\
					&\le 1^2 \ E\big[ \Abs{g(X_T) - g(X^n_N)}^2 \big]					
						\le C^X \abs{\pi^n} .
			\end{align*}
		Theorem \ref{theorem---global.error.estimate} thus applies to the numerical approximation of BSDE \eqref{equation---reference--BSDE} with terminal condition $\xi^{L_n}=T^{L_n}(\xi)$ 
		and we have, for any $\epsilon > 0$,
			\begin{align*}
				\max_{i = 0 \ldots N^n} E\big[ \abs{Y^{t}_\tni - Y^{n,t}_i}^2 \big] 
				\quad \text{and} \quad  
				E\bigg[ \sum_{i=0}^{N^n} \Abs{\wb Z^t_\tni - Z^{n,t}_i}^2 \hnip \bigg] 
						\le e^{(\epsilon+\wt M_y^{g^{L_n}}(\abs{\pi^n})) T} \ C_\epsilon \, \abs{\pi^n} .
			\end{align*}
		Now, we have
			\begin{align*}
				\wt M_y^{g^{L_n}}(\abs{\pi^n}) 
					&= M_y + 2 L_z^2 + 3 (d+1) L_y^2 \hnip (1 + 2 \ \norm{g^{L_n}}_\infty^{2(m-1)})																\\
						& =  M_y + 2 L_z^2 + 3 (d+1) L_y^2 \hnip (1 + 2 \ (L_n)^{2(m-1)})																					\\
						& =  M_y + 2 L_z^2 + 3 (d+1) L_y^2 \hnip (1 + 2 \ (L_0)^{2(m-1)} \ n^{\frac{\alpha}{2(m-1)} 2(m-1)} )								\\
						&\le M_y + 2 L_z^2 + 3 (d+1) L_y^2 T (n^{-1} + 2 \ (L_0)^{2(m-1)} \ n^{\alpha-1} )	.
			\end{align*}
		This guarantees the non-explosion of $(\epsilon+\wt M_y^{g^{L_n}}(\abs{\pi^n})) T$ as $n \goesto +\infty$. 
		In fact, if $\alpha < 1$, we obtain again $\wt M_y^{g^{L_n}}(\abs{\pi^n}) \goesto M_y + 2L_z^2 \le M_y/2$.
			
		Finally, using the triangle inequality in the form $(a-c)^2 \le 2 (a-b)^2 + 2 (b-c)^2$ completes the proof.
	\end{proof}

\section{Extension to overall-monotone drivers}			\label{section---extensions.to.overall-monotone.drivers}

This section is devoted to a generalization on the type of driver allowed in the scalar case, and relaxes the assumption that \MonY is satisfied, everywhere, with a $M_y < 0$.
Indeed, there are two important limitations with this which can occur, for instance, with drivers of FitzHugh--Nagumo type : $f(y,z) = c y^3 + b y^2 + a y$.
In the simple case $f(y) = -y^3$, we see that the derivative is everywhere negative, but null at the origin. So we can only have $M_y=0$ in \MonY. Yet the driver is still strictly decreasing on $\R$. 
In the case $f(y) = y - y^3$, we see that the derivative has a maximum of $1$, so we can only have $M_y = 1$ in \MonY.
However, in both case, when $y$  is far enough from the origin, we do have $y f(y) \le M_y \abs{y}^2$ for some $M_y < 0$ : large values of $y$ result in a drift pointing toward the origin.
This guarantees a certain size-stability for the continuous-time dynamics.
We explain in this section how to adapt the ATS scheme to these situations and preserve the numerical stability.

\paragraph*{}
So in this section, \RegY, \LipZ and \GrowthY are still assumed but not \MonY.
We come back to the setting of a bounded $\xi$, in the scalar case $k=1$.
The condition \MonY is replaced by a generalized monotone growth assumption: 
we assume that $f(0,0)=0$ and that for some $K > 0$ and $M_y < 0$,
	\begin{itemize}
		\item [] \GMonGr  for all $y$ : $\abs{y} > K$, we have ${y}{f(y,0)} \le M_y \abs{y}^2$.
	\end{itemize}
Remark that the assumption $f(0,0)=0$ is more for convenience than it is a restriction. 
First, \GMonGr implies that $\lim_{+\infty} f(y,0)=-\infty$ and $\lim_{-\infty} f(y,0)=+\infty$. 
Then, if $f(y_0,0)=0$, one can do the change of variable $(\overrightarrow{Y},\overrightarrow{Z}):=(Y-y_0,Z)$ and have 
$\overrightarrow{f}(\overrightarrow{y},\overrightarrow{z}):=f(\overrightarrow{y}+y_0,\overrightarrow{z})$ satisfy $\overrightarrow{f}(0,0)=0$.

\paragraph*{}
Let us first note that it can be seen from proposition 2.2 and theorem 2.2 in \cite{Pardoux1999} that the BSDE \eqref{equation---reference--BSDE} is well-posed in this setting, 
by remarking that \LipZ, \GMonGr and \RegY imply that for all $y,z$ we have, denoting $b = M_y + L_y (1+K^{m-1}) + L_z^2$,
	\begin{align*}
		y \, f(y,z) = y \, f(y,0) + y \, \big( f(y,z)-f(y,0) \big)
			\le b \abs{y}^2 + \frac{1}{4} \abs{z}^2,
	\end{align*}
which is the only estimate required.
Now, we have the following finer estimate on the size of the $Y$-component of the solution. 

	\begin{proposition}
		We have for all $t \in [0,T]$
			\begin{align*}
				\abs{Y_t} \le \max\big( e^{M_y (T-t)} \Norm{\xi}_\infty , K \big).
			\end{align*}
	\end{proposition}

This result is only about continuous-time BSDEs. However, since it does not appear to be present in the literature on continuous-time BSDEs, we provide a short proof for it.

	\begin{proof}
		We only prove the upper bound, the lower bound $Y_t \ge \min\big( -e^{M_y (T-t)} \Norm{\xi}_\infty , -K \big)$ being obtained similarly.
		$(Y,Z)$ solves the BSDE with terminal condition $\xi$ and driver $f$. 
		Denoting by $(y^+,z^+)$ the solution to the BSDE with terminal condition $\Norm{\xi}_\infty$ and driver $f$, 
		since $\xi \le \Norm{\xi}_\infty$, we have by the comparison theorem 2.4 in \cite{Pardoux1999} that for all $t \in [0,T]$, a.s., $Y_t \le y^+_t$.
		Since $f$ and $\Norm{\xi}_\infty$ are deterministic, $z^+=0$ and $y^+$ in fact solves the ODE
			\begin{align*}
				dy^+_t = - \varphi(y^+_t) dt 
				\qquad \text{with} \qquad  
				y^+_T = \Norm{\xi}_\infty ,
			\end{align*} 
		where we denote by $\varphi$ the function $f(\cdot,0)$. 

		Since \GMonGr implies that $\varphi(y) < 0$ for all $y > K$, we know that if there exists $t \in [0,T]$ such that $y^+_t \le K$, then we have for all $s \le t$, $y^+_s \le K$.
		We then have two cases. 
		If $y^+_T = \Norm{\xi}_\infty \le K$, then the sought bound holds because $y^+_t \le K$ for all $t$.
		So let us consider the case $y^+_T = \Norm{\xi}_\infty > K$.
		Then, denote by $\beta_t = 1_{\{ y^+_t \neq 0 \}} \frac{ \varphi(y^+_t) }{y^+_t}$. We have, for all $t \in [0,T]$,
		$y^+_t = \exp\big({ \int_t^T \beta_u du}\big) \Norm{\xi}_\infty$.
		Since $\beta_u \le M_y < 0$ when $y^+_u > K$ by \GMonGr, we have, as long as $y^+_t > K$ holds, $y^+_t \le e^{M_y (T-t)} \Norm{\xi}_\infty$.
		Thus, for all $t \in [0,T]$, $Y_t \le y^+_t \le \max\big( e^{M_y (T-t)} \Norm{\xi}_\infty , K \big)$.		
	\end{proof}

We now want an adapted time-steps partition so that the scheme preserves this bound.
The idea is to rely on the comparison property. 
To this end, instead of the standard grid $\pi^n$ given by \eqref{equation---reference--definition.of.hnip.1dim},
we use the time-grid $\pi^{n,+} = (\tni)_{i= 0 \ldots N^{n,+}}$ 
defined like $\pi^n$ but with time-steps closer to \eqref{equation---reference--definition.of.hnip.1dim.with.comparison}. 
More precisely, they are defined by
	\begin{align}		\label{equation---reference--definition.of.hnip.1dim.overall-monotone.case}
		h^{n,+}_\ip := \min \left\{ \frac{1}{3L_y \big( 1 + 2 \max\big( e^{(m-1) M_y (T-t^n_\ip)} \Norm{\xi}_\infty^{m-1}, K^{m-1} \big) \big)} 
					\ , \ h_0(L_z) \ , \ \hmax \ , \ \tnip \ , \ \frac{T}{n} \right\} .
	\end{align}
The approximations produced are still denoted by $(Y^n_i,Z^n_i)_{i = 0 \ldots N}$.

\paragraph*{}
The main proposition is the following.

	\begin{proposition}
		We have, for all $i \in \{0, \ldots, N^{n,+} \}$, 
			\begin{align*}
				\abs{Y^n_i} \le \max \big( e^{ M_y (T-\tni)} \Norm{\xi}_\infty , K \big) .
			\end{align*}
	\end{proposition}

Consequently, the scheme is indeed numerically stable, as desired.

	\begin{proof}
		First off, notice that  
		the first part of proposition \ref{proposition---stability.estimate.one-step.via.linearization.1d} (the fact that one can write $\delta Y_i = E[\delta Y_\ip \E_\ip]$) 
		still holds in this generalized setting, since its proof only relies on \RegY and \LipZ, and does not use \MonY.
		We will prove the sought bound by backward induction, relying one the one-step comparison property that follows from \ref{proposition---stability.estimate.one-step.via.linearization.1d}.
		
		The bound is clearly true for $i=N^{n,+}$. Let us now assume it holds for $i+1$, for $i \in \{0, \ldots, N^{n,+}-1 \}$.
		Define the constant $y^+_\ip = \max \big( e^{M_y (T-\tnip)} \Norm{\xi}_\infty , K \big)$.
		We see when taking the power $m-1$ that
			\begin{align*}
				\Norm{Y^n_\ip}_\infty^{m-1} 
				\quad \text{and} \quad 
				\Norm{y^+_\ip}_\infty^{m-1} = (y^+_\ip)^{m-1}
				\le \max \big( e^{ (m-1) M_y (T-\tnip)} \Norm{\xi}_\infty^{m-1} , K^{m-1} \big) ,
			\end{align*}
		so we have
			\begin{align*}
				h^{n,+}_\ip 
					&= \min \left\{ \frac{1}{3L_y \big( 1 + 2 \max\big( e^{(m-1) M_y (T-t^n_\ip)} \Norm{\xi}_\infty^{m-1}, K \big) \big)} 
						\ , \ h_0(L_z) \ , \ \hmax \ , \ \tnip \ , \ \frac{T}{n} \right\}																													\\
					&\le \min \left\{ \frac{1}{3L_y \big( 1 + \Norm{Y^n_\ip}_\infty^{m-1} + \Norm{y^+_\ip}_\infty^{m-1} \big)} 
						\ , \ h_0(L_z) \ , \ \hmax \ , \ \tnip \ , \ \frac{T}{n} \right\}																													\\
					& = h^0_\ip(,\Norm{Y^n_\ip}_\infty,\Norm{y^+_\ip}_\infty) ,
			\end{align*}
		where $h^0_\ip(,\Norm{Y^n_\ip}_\infty,\Norm{y^+_\ip}_\infty)$ is defined in proposition \ref{proposition---stability.estimate.one-step.via.linearization.1d}.
		Then, applying said proposition \ref{proposition---stability.estimate.one-step.via.linearization.1d} 
		with $Y_\ip = Y^n_\ip$ and $\wh Y_\ip = y^+_\ip$ (which is deterministic),
		and denoting by $(\wh Y_i,\wh Z_i)$ the output of scheme \eqref{equation---reference--ATS.scheme} for $\wh Y_\ip$,
		it follows from $Y^n_\ip \le y^+_\ip$ that $Y^n_i \le \wh Y_i$.
		It remains to prove that $\wh Y_i \le \max \big( e^{ M_y (T-\tni)} \Norm{\xi}_\infty , K \big)$ in order to complete the proof.

		Let us remark that $\wh Z_i = E_i[y^+_\ip \Hip^*] = 0$ since $y^+_\ip$ is constant and $E_i[\Hip] = 0$.
		So, denoting by $\varphi(\cdot) = f(\cdot,0)$, we have $\wh Y_i = y^+_\ip + \varphi(y^+_\ip) h^{n,+}_\ip$.
		The estimate for $\wh Y_i$ then depends on the size (or location) of $y^+_\ip$.
		
		\emph{Case 1 : $y^+_\ip > K$.} So in fact $y^+_\ip = e^{M_y (T-\tnip)} \Norm{\xi}_\infty > K$.
		Then, setting $\beta_\ip = \varphi(y^+_\ip)/y^+_\ip$, we have $\beta_\ip \le M_y$ by \GMonGr and therefore
			\begin{align*}
				\wh Y_i = \big( 1 + \beta_\ip h^{n,+}_\ip \big) y^+_\ip \le e^{M_y h^{n,+}_\ip} y^+_\ip = e^{M_y h^{n,+}_\ip} e^{M_y (T-\tnip)} \Norm{\xi}_\infty =  e^{M_y (T-\tni)} \Norm{\xi}_\infty,
			\end{align*}
		which proves the desired bound : $\wh Y_i \le \max \big( e^{ M_y (T-\tni)} \Norm{\xi}_\infty , K \big)$.
		
		\emph{Case 2 : $y^+_\ip = K$.} Then, we have $\wh Y_i = K + \varphi(K) h^{n,+}_\ip$.
		But \GMonGr implies, when taking the limit $y \goesto K$ from above, that $\varphi(K) < 0$. 
		We thus have the desired bound since $\wh Y_i < K = \max \big( e^{ M_y (T-\tni)} \Norm{\xi}_\infty , K \big)$.
		
		The lower bound $Y^n_i \ge -\max  \big( e^{ M_y (T-\tni)} \Norm{\xi}_\infty , K \big)$ is obtained in a similar fashion, 
		applying proposition \ref{proposition---stability.estimate.one-step.via.linearization.1d} 
		with inputs $Y_\ip = Y^n_\ip$ and $\wh Y_\ip = y^-_\ip := -y^+_\ip$.
	\end{proof}

\section{Numerical simulations}			\label{section---numerical.simulations}

In this section, we illustrate numerically the results of sections \ref{section---numerical.stability} and \ref{section---convergence.of.the.scheme}.

\paragraph*{}
To implement the Adapted Time-Steps scheme, one also needs to numerically approximate the conditional expectations appearing in 
\eqref{equation---reference--X.scheme-time.discretization.only}-\eqref{equation---reference--ATS.scheme}. 
To this end, we use a quantization method (see \cite{ChassagneuxRichou2016} and \cite{PagesSagna2015} for details). 
For given strictly positive $(\eta_i)_{i=0 \ldots N}$ and integers $(m_i)_{i = 0 \ldots N}$, we consider the spatial grids
	\begin{align*}
		\Gamma_i = x_0 + \Big\{  x \in \eta_i \Z^d \ \Big|\ \abs{ x^c} \le m_i \eta_i, \text{ for all } c  = 1 \ldots d \Big\} .
	\end{align*}
We denote by $\Pi_i$ the projection from $\R^d$ on the grid $\Gamma_i$ in the distance associated with the $L^\infty$-norm (i.e. component-wise projection).
The forward process is then approximated as follows. It is initialized with $\wb{X}_0 = x_0$ and then, for $i=0 \ldots N-1$,
	\begin{align}		\label{equation---reference--X.scheme--fully.discretized}
		\wb X_\ip = \Pi_i \Big( \wb X_i + \mu(\ti,\wb X_i) \hip + \sigma(\ti,\wb X_i) {\Delta W}_\tip \Big).
	\end{align}
For the backward processes, we initialize the scheme with $\wb Y_N = \xi^N = g(\wb X_N)$, $\wb Z_N = 0$ and then, for $i=N-1$ down to $0$, $\wb Y_\ip$ is used to compute $(\wb Y_i,\wb Z_i)$ as
	\begin{align}		\label{equation---reference--ATS.scheme--fully.discretized} 
		\left\{\begin{aligned}
			\wb Y_i &= E_i\Big[ \wb Y_\ip + f(\wb Y_\ip,\wb Z_i) \hip \Big]											\\
			\wb Z_i &= E_i\Big[ \wb Y_\ip \Hip^* \Big] .
		\end{aligned}\right.
	\end{align}

$(\wb Y_i,\wb Z_i)_{i=0 \ldots N}$ is a Markov process and, for all $i \in \{0, \ldots , N\}$, there exist functions $y_i$ and $z_i$, from $\Gamma_i$ to $\R$ and $\R^{1 \times d}$ respectively, such that $\wb Y_i = y_i(\wb X_i)$ and $\wb Z_i = z_i(\wb X_i)$.
We denote by $x^i_k = x_0 + k \, \eta_i$ and by $x^\ip_l = x_0 + l \, \eta_\ip$, and we denote the transition probabilities $p^i_{k,l} = P\big( \wb X_\ip = x^\ip_l \big| \wb X_i = x^i_k \big)$, which can be obtained, for instance, by Monte Carlo estimation. We have
	\begin{align*}
		\left\{\begin{aligned}
			y_i(x^i_k) &= \sum_{l \,:\, x^\ip_l \in \Gamma_\ip} p^i_{k,l} \ \Big( y_\ip(x^\ip_l) + f(y_\ip(x^\ip_l),z_i(x^i_k)) \, \hip \Big)			\\
			z_i(x^i_k) &= \sum_{l \,:\, x^\ip_l \in \Gamma_\ip} p^i_{k,l} \ y_\ip(x^\ip_l) \ \frac{ T^{R(\hip)} \big(  ( x^\ip_l - x^i_k - \mu(\ti,x^i_k) \hip ) / \sigma(\ti,x^i_k) \big) }{\hip} .
		\end{aligned}\right.
	\end{align*}

\subsection{Numerical stability}

We look at the example where $f(y) = -y-y^3$ and $g(x) = \min(\abs{x}^2,c)$ for some constant $c$. 
This driver $f$ satisfies \MonY, \RegY, \LipZ with $m=3$, $M_y=-1$, $L_y = 3/2$ and $L_z=0$.
For the forward process, we take a standard Brownian motion, $X=W$ (so $x_0=0$, $\mu=0$ and $\sigma=1$). 

Having computed the partition $\pi^{n,+}$ with time-steps $\hip$ given by \eqref{equation---reference--definition.of.hnip.1dim.with.comparison}, we define for all $i$ the upper bounds $x^i_{\max} = 0 + 0 \, . \, \ti + 1 \, . \,  q \, \sqrt{\ti}$, and lower bounds $x^i_{\min} = -x^i_{\max}$, with $q=3$ (so that $P(\abs{G}>q) \approx 0$, whenever $G \sim \cN(0,1)$). We use $m_i=i$ and $\eta_i = \frac{x^i_{\max}-x^i_{\min}}{2m_i}$.

\paragraph*{}
In order to observe the preservation of size bounds, we plot for each $\ti$ the maximum and minimum of $\wb Y_i$. 

Figure \ref{Figure---numerical.stability} shows the results for the implicit and explicit schemes with regular time-grid, as well as for the ATS grid, when $n=15$, for the cases $c=3.6$, $c=4$ and $c=6$.
For the ATS scheme, we report the number $N^n$ of time-steps in the computed partition, 
its excess cardinal, $N^n/n$,
and its non-uniformity, defined as the ratio of the biggest time-step over the smallest time-step.

\begin{table}[htb]	
	\centering 
	\begin{tabular}{|c|c|c|c|}
		\hline 
		c & 3.6 & 4 & 6 \\
		\hline 
		$N^n$ & 15 & 17 & 27 \\ 
		\hline 
		$N^n/n$ & 1 & 1.13 & 1.8 \\
		\hline 
		non-uniformity & 1.4 & 1.7 & 3.7 \\
		\hline 
	\end{tabular}
	\caption{ Characteristics of the partition $\pi^{n,+}$. } 
\end{table}

\begin{figure}[h]
	\centering
	\subfigure [Case c=3.6]
	{	\includegraphics[scale=0.363]{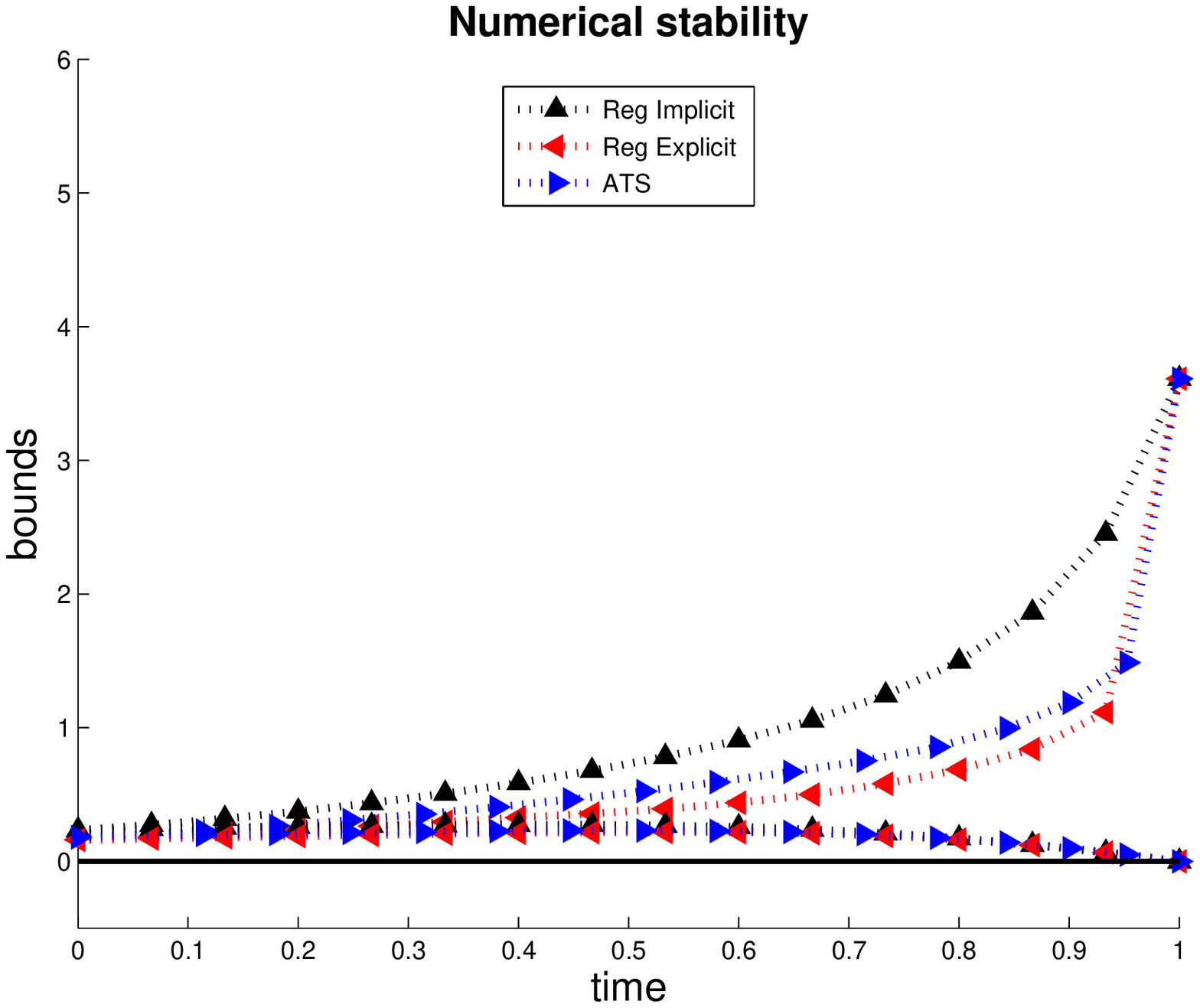}	
		\label{Figure---numerical.stability.c=3,61}
	}
	\subfigure [Case c=4]
	{	\includegraphics[scale=0.363]{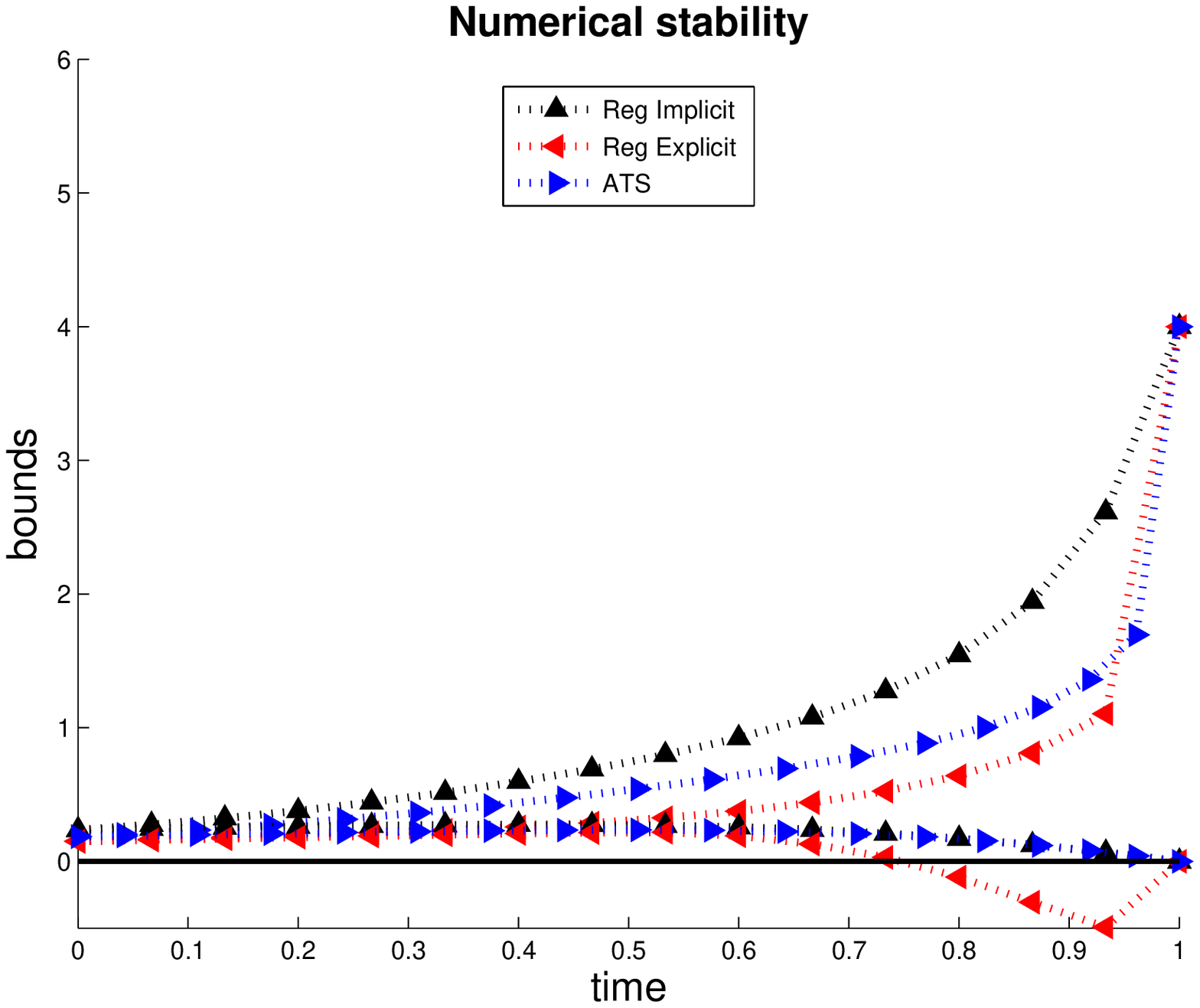}	
		\label{Figure---numerical.stability.c=4}
	}
	\subfigure [Case c=6]
	{	\includegraphics[scale=0.38]{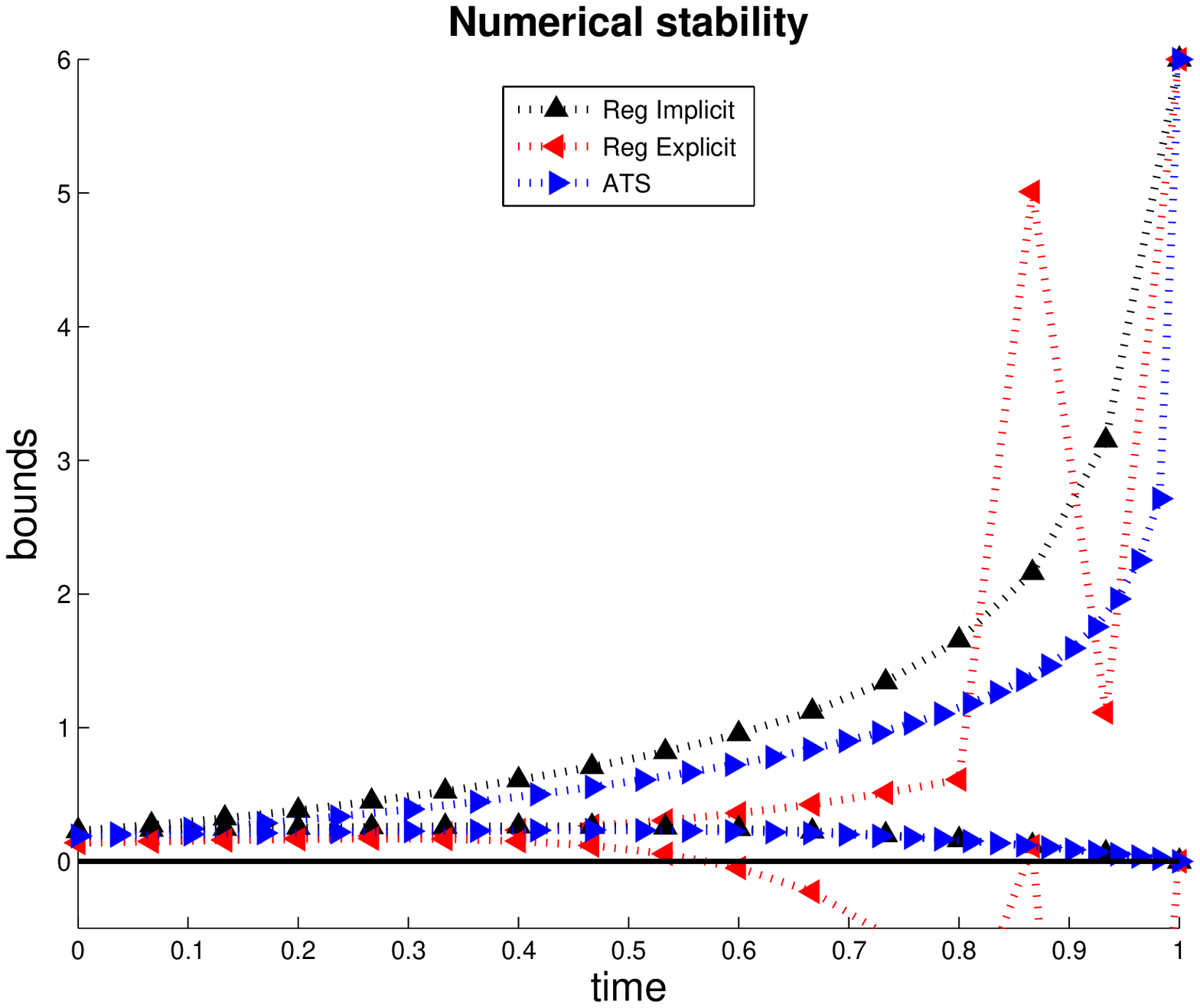}	
		\label{Figure---numerical.stability.c=6}	
	}		
	\caption{ Numerical stability. }
	\label{Figure---numerical.stability}
\end{figure}

\paragraph*{}
In each case, we have a different value of $\Norm{g}_\infty = c$, leading to a different time-grid $\pi^{n,+}$.
Here, the upper bounds give the norm and we see that the ATS scheme is always strongly numerically stable (it is designed for that), as is the implicit scheme. Also, from the comparison property and the fact that the terminal condition $\xi=0$ leads to the solution $0$, we know that the numerical approximations should remain positive.
When $c=3.6$, the standard explicit scheme is stable as well. Here, $\pi^{n,+}$ has as many time-steps as the regular time-grid, $N^n=n=15$, however the non-uniformity ratio of 1.4 indicates that they are not spread evenly across $[0,T]$.
As $c$ is increased to $4$ and then $6$, we see that the standard explicit scheme fails to be stable and to preserve positivity. 
The ATS grid handles the increased values and increased terms $f(\wb Y_\ip,\wb Z_i) \hip$ of the dynamics by placing more dates and increasing the non-uniformity : the time-steps are small when $\norm{\wb Y_\ip}_\infty$ is large, to guarantee stability, and then bigger when it is not computationally necessary to take smaller time-steps.

\subsection{Errors and computation time}

We look here at an example with still $X=W$ but with $f(y) = -y^3$ and $g(x) = \min(\abs{x}^2,c)$ with constant $c=7$. 
This driver $f$ satisfies \MonY, \RegY and \LipZ  with $m=3$, $M_y=0$, $L_y = 3/2$ and $L_z=0$. 
Because $M_y=0$, instead of using the time-steps $\hip$ given by \eqref{equation---reference--definition.of.hnip.1dim.with.comparison}, 
we should use here the partition designed in section \ref{section---extensions.to.overall-monotone.drivers}, with time steps given by \eqref{equation---reference--definition.of.hnip.1dim.overall-monotone.case}. 
For this, we choose to take $K=1$, and find that \GMonGr then holds with $M_y=-K^2 < 0$.

\begin{figure}[ht]
	\centering
	{	\includegraphics[scale=0.45]{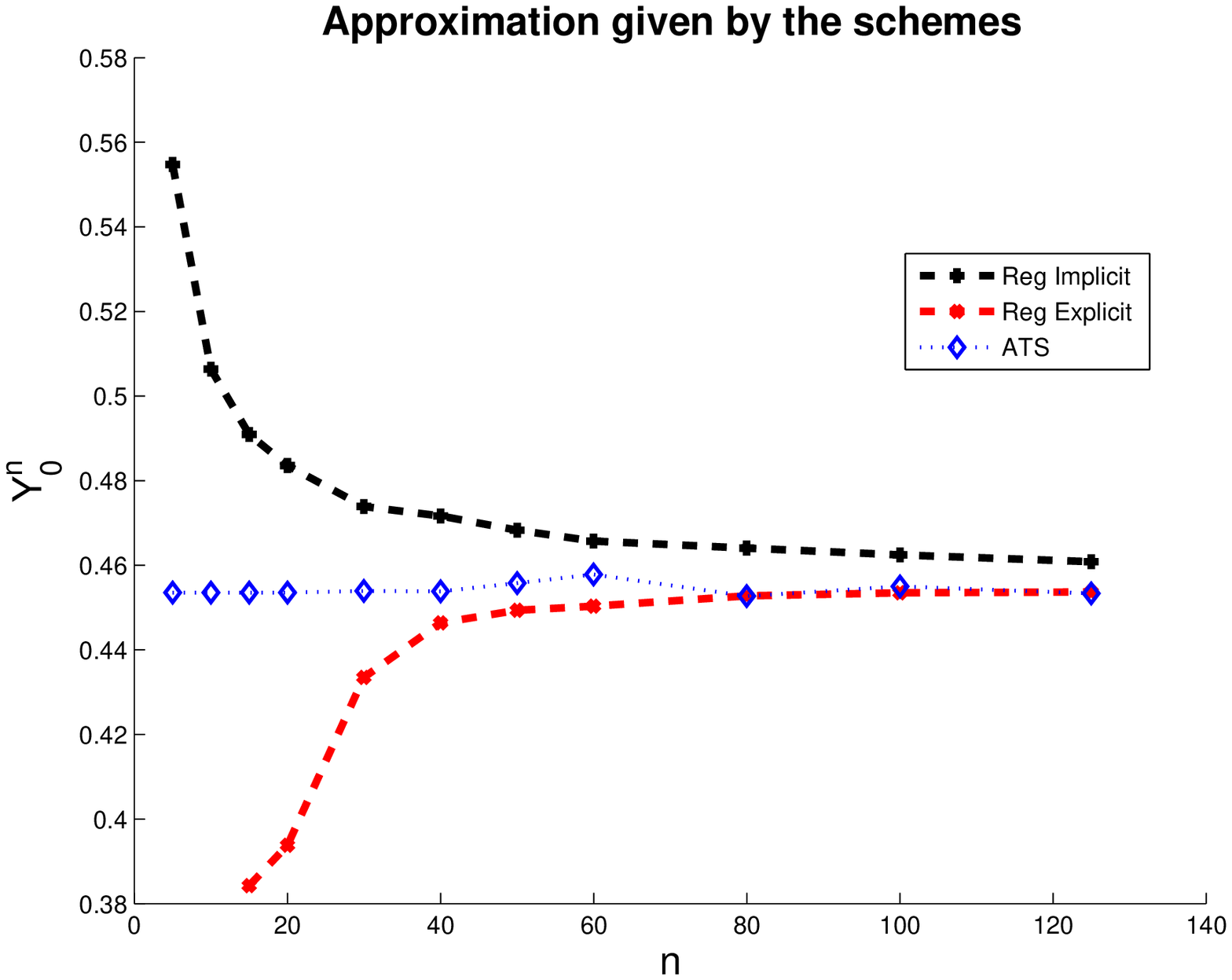}		
	}
	\caption{ Behaviour of the schemes ; $n \in [5,125]$. }
	\label{Figure---errors.Y0.vs.n}
\end{figure}

\paragraph*{}
On Figure \ref{Figure---errors.Y0.vs.n} we plot the value $Y^n_0$ obtained as a function of $n$. 
We observe that the standard implicit and explicit scheme tend to respectively overshoot and undershoot, which is understandable since the driver $f$ is decreasing. 
The ATS scheme returns value within the implicit-explicit interval, before finally merging with the regular explicit scheme when its ATS partition becomes a regular one.
This initial behaviour is somewhat reminiscent of that of the trapezoidal scheme in \cite{LionnetDosReisSzpruch2015} although the ATS scheme is entirely explicit. The reason why it does not undershoot as much as the standard explicit scheme is that it takes smaller time steps when the solution, and thus $f(\wb Y_\ip,\wb Z_i)$, is too big.

Figure \ref{Figure---errors.and.computation.time} on the other hand plots the error versus the computational time. The error is taken with respect to the average of the upper and lower values for $n=150$.

\begin{figure}[ht]
	\centering
	{	\includegraphics[scale=0.40]{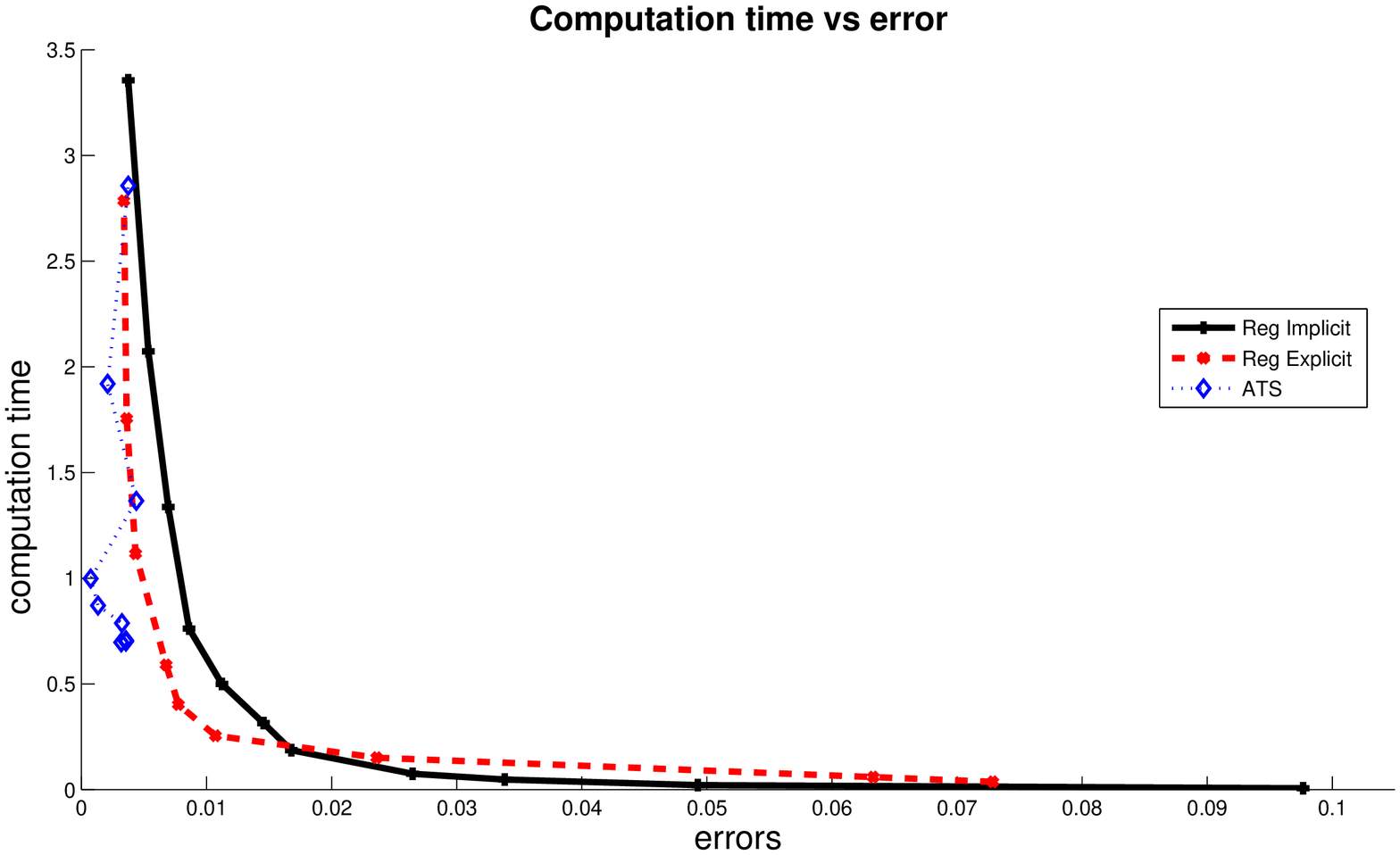}		
	}
	\caption{ Behaviour of the schemes ; $n \in [5,125]$. }
	\label{Figure---errors.and.computation.time}	
\end{figure}

\paragraph*{}
We see that the ATS therefore realizes a ``continuation by numerical stability'' of the explicit scheme from the range of high $n$ into the range of low $n$. It overall benefits from the lower cost of the standard explicit scheme (compared to the regular implicit scheme, for equal $N$), while avoiding explosion problems. In the range of low $n$'s, the ATS scheme computes a partition with more points ($N^n \ge n$), thus incurring a higher cost. It can however remain competitive with the implicit scheme because it allocates the computational effort more relevantly over the time interval. 

On the other hand, we note that the computation of the time-grid $\pi^n$ requires some knowledge of $g$ (for $\norm{\xi}_\infty$), 
as well as some detailed knowledge of $f$ (not merely the degree $m$ of the polynomial growth, but also (upper bounds for) the constants $M_y$, $L_y$ and $L_z$).
It thus requires more input (to be estimated by hand) when changing the driver and terminal condition, but then makes better use of the structure of the driver.

\appendix 
\section{Appendix}		\label{section---apprendix}

\subsection*{Proof of the estimate \eqref{equation---reference--bound.on.error.from.truncating.H}}

We provide here the proof that
	\begin{align*}	
		\max_{i = 0 \ldots N-1} E\left[ \Abs{ \frac{\Delta W_\tip}{\hip} - \Hip }^2 \right]	 \le C^H .
	\end{align*}

	\begin{proof}
		Denoting $\delta_\ip = \frac{\Delta W_\tip}{\hip} - \Hip$ and $G_\ip = \Delta W_\ip / \sqrt{\hip} \sim \cN(0,1)$, we have
			\begin{align*}
				\delta_\ip &= \frac{\Delta W_\tip}{\hip} - \frac{T^{R(\hip)} (\Delta W_\tip)}{\hip} 
							= \frac{G_\tip}{\sqrt{\hip}} - \frac{T^{r(\hip)} (G_\tip)}{\sqrt{\hip}} 											\\
					&=\frac{1}{\sqrt{\hip}} \big( G_\ip - T^{r(\hip)} (G_\ip) \big) \ 1_{\{\abs{G_\ip}_\infty > r(\hip)\}} ,	
			\end{align*}
		recalling that $R(\hip) = \sqrt{\hip} r(\hip)$.
		
		We now consider $h>0$ and a random variable $G \sim \cN(0,1)$. Using the fact that $\abs{x-T^{r(h)}(x)} \le \abs{x}$ and $\abs{x}_\infty \le \abs{x}$, we have
			\begin{align*}
				\abs{\delta_h} = \frac{1}{\sqrt{h}} \abs{G - T^{r(h)}(G)} \ 1_{\{\abs{G}_\infty > r(h)\}} \le \frac{1}{\sqrt{h}} \abs{G} \ 1_{\{\abs{G} > r(h)\}} .
			\end{align*}
		Using the Gaussian tail estimate 
			\begin{align*}
				\int_{\abs{x}>r} \abs{x}^2 \frac{e^{-\abs{x}^2/2}}{(2\pi)^{d/2}} dx \le C_{2,d} \ r^{d} e^{-\abs{r}^2/2} ,
			\end{align*}
		for some constant $C_{2,d} > 0$, we obtain
			\begin{align*}
				E[\abs{\delta_h}^2] \le \frac{1}{h} E\big[\abs{G}^2 \ 1_{\{\abs{G} > r(h)\}} \big] \le C_{2,d} \ \frac{1}{h} r(h)^{d} e^{-r(h)^2/2} = C_{2,d} \ \varphi(h).
			\end{align*}
		
		Given that $r(h) = \sqrt{2} \ln(1/h)$, we have
			\begin{align*}
				\varphi(h) = 2^{d/2} \ \frac{1}{h} \ln\Big(\frac{1}{h}\Big)^d \ e^{-\ln(1/h)^2} 		
						= 2^{d/2} \exp\bigg( \ln\Big(\frac{1}{h}\Big) + d \ln \ln\Big(\frac{1}{h} \Big) - \ln \Big(\frac{1}{h}\Big)^2 \bigg) .
			\end{align*}
		Since $h \in ]0,\hmax]$, $\varphi$ is continuous, and $\lim_{h\goesto 0} \varphi(h) = 0$, we see that $\varphi$ and therefore the upper bound for $E[\abs{\delta_h}^2]$ is bounded.
		Hence 
			\begin{align*}	
				\max_{i = 0 \ldots N-1} E\left[ \Abs{ \frac{\Delta W_\tip}{\hip} - \Hip }^2 \right]	 \le C^H .
			\end{align*}

	\end{proof}

\subsection*{Proof of proposition \ref{proposition---estimation.of.the.discretization.error}}

We provide here the proof of proposition \ref{proposition---estimation.of.the.discretization.error}
It is is split in two parts. 

	\begin{proof}[Proof of the estimate for the $Z$-component in Proposition \ref{proposition---estimation.of.the.discretization.error}] 
		First recall that from the martingale increment property of $H_\ip$ we have
			\begin{align*}
				E_i\left[ \int_\ti^\tip Z_u dW_u \ H_\ip^* \right] 
					= E_i\left[ \Big( Y_\tip + \int_\ti^\tip f(u,Y_u,Z_u) du \Big) \ H_\ip^* \right] \ .
			\end{align*}
		We write 
			\begin{align*}
				\overline{Z}_\ti - \widehat{Z}_i 
					&= E_i\Bigg[ \int_\ti^\tip Z_u dW_u \ \frac{\Delta W_\tip^*}{\hip} \Bigg] - E_i\Bigg[ Y_\tip H_\ip^* \Bigg]																						\\
					&= E_i\Bigg[ \int_\ti^\tip Z_u dW_u \ \frac{\Delta W_\tip^*}{\hip} \Bigg] - E_i\Bigg[ \int_\ti^\tip Z_u dW_u \ H_\ip^* \Bigg]															\\
							& \qquad + E_i\Bigg[ \Big( Y_\tip + \int_\ti^\tip f(Y_u,Z_u) du \Big) \ H_\ip^* \Bigg] - E_i\Bigg[ Y_\tip H_\ip^* \Bigg] 															\\
					&= E_i\Bigg[ \int_\ti^\tip Z_u dW_u \bigg(\frac{\Delta W_\tip}{\hip} - H_\ip \bigg)^* \Bigg]
							   + E_i\Bigg[ \int_\ti^\tip f(Y_u,Z_u) du \ H_\ip^* \Bigg] .
			\end{align*}
		We now take the square and use the Cauchy--Scwhartz inequality and the Itô isometry : 
			\begin{align*}
				\Abs{\overline{Z}_\ti - \widehat{Z}_i}^2			
					&\le 2 E_i\bigg[ \Abs{\int_\ti^\tip Z_u dW_u}^2 \bigg] E_i\bigg[ \Abs{\frac{\Delta W_\tip}{\hip} - H_\ip }^2 \bigg]																			\\
							&\hspace{1cm} + 2 E_i\bigg[ \Abs{\int_\ti^\tip f(Y_u,Z_u) du}^2 \bigg] E_i\Big[ \abs{\ H_\ip}^2 \Big] 																						\\
					&\le 2 E_i\bigg[ \int_\ti^\tip \abs{Z_u}^2 du \bigg] E_i\bigg[ \Abs{\frac{\Delta W_\tip}{\hip} - H_\ip }^2 \bigg]																							
							+ 2 E_i\bigg[ \hip \int_\ti^\tip \abs{f(Y_u,Z_u)}^2 du \bigg] \frac{\Lambda_i d}{\hip} 																												\\
					& = 2 E_i\bigg[ \int_\ti^\tip \abs{Z_u}^2 du \bigg] E\bigg[ \Abs{\frac{\Delta W_\tip}{\hip} - H_\ip }^2 \bigg]																							
							+ 2 \Lambda_i d E_i\bigg[ \int_\ti^\tip \abs{f(Y_u,Z_u)}^2 du \bigg] ,
			\end{align*}
		since $\frac{\Delta W_\tip}{\hip} - H_\ip$ is independent from $\F_i$.
		Hence, taking expectations, multiplying by $\hip \le \abs{\pi}$, using $\Lambda_i \le 1$ and summing, we obtain
			\begin{align*}
				\sum_{i=0}^{N-1} E\big[ \abs{\overline{Z}_\ti - \widehat{Z}_i}^2 \big] \hip
					&\le 2 \, \abs{\pi} \, E\bigg[ \int_0^T \abs{Z_u}^2 du \bigg] \max_{i=0 \ldots N-1} E\bigg[ \Abs{\frac{\Delta W_\tip}{\hip} - H_\ip }^2 \bigg]							\\
						&\qquad + 2 d \abs{\pi} \ E\bigg[ \int_0^T \Abs{f(Y_u,Z_u)}^2 du \bigg] .
			\end{align*}
		The fact that this is bounded by some $C^Z \abs{\pi}$ then follows from \eqref{equation---reference--bounds.on.X.Y.Z} 
		and from \eqref{equation---reference--bound.on.error.from.truncating.H}.
	\end{proof}

	\begin{proof}[Proof of the estimate for the $Y$-component of Proposition \ref{proposition---estimation.of.the.discretization.error}] 
		We first write
			\begin{align*}
				Y_\ti - \widehat{Y}_i 
					&= E_i\bigg[ \int_\ti^\tip f(Y_u,Z_u) du - f(\ti,Y_\tip,\widehat{Z}_i) \hip \bigg]					 																												\\
					&= E_i\bigg[ \int_\ti^\tip f(Y_u,Z_u) - f(Y_\tip,Z_u) du \bigg]																																									
								+ E_i\bigg[ \int_\ti^\tip f(Y_\tip,Z_u) - f(Y_\tip,\overline{Z}_\ti) du \bigg]																																\\
							& \qquad 	+ E_i\bigg[ f(Y_\tip,\overline{Z}_\ti) - f(\ti,Y_\tip,\widehat{Z}_i) \bigg] \hip .
			\end{align*}	
		We now take the square and use the Cauchy--Scwhartz inequality, \RegY and \LipZ :
			\begin{align*}
				\abs{Y_\ti - \widehat{Y}_i}^2
					&\le 3 E_i\bigg[ \hip \int_\ti^\tip \abs{ f(Y_u,Z_u) - f(Y_\tip,Z_u) }^2 du \bigg]																																	\\
							&\qquad + 3 E_i\bigg[ \hip \int_\ti^\tip \abs{ f(Y_\tip,Z_u) - f(Y_\tip,\overline{Z}_\ti) }^2 du \bigg]																							\\
							&\qquad	+ 3 E_i\Big[ \abs{ f(Y_\tip,\overline{Z}_\ti) - f(\ti,Y_\tip,\widehat{Z}_i) }^2 \Big] \hip^2 																								\\
					&\le 3 \hip E_i\bigg[ \int_\ti^\tip L_y^2  (1+\abs{Y_u}^{m-1}+\abs{Y_\tip}^{m-1})^2 \Abs{Y_u - Y_\tip}^2 du \bigg]																		\\
							&\qquad + 3 \hip E_i\bigg[  \int_\ti^\tip L_z^2 \abs{Z_u - \overline{Z}_\ti}^2 du \bigg]
									+ 3 \hip^2 E_i\big[ L_z^2 \abs{\overline{Z}_\ti - \widehat{Z}_i}^2 \big] .
			\end{align*}	
		Taking expectation and using further a Cauchy--Schwartz on the $Y$-term, we have
			\begin{align*}
				E\big[ \abs{Y_\ti - \widehat{Y}_i}^2 \big]
					&\le 3 L_y^2 \hip E\Big[ \hip (1 + 2 \sup_{\ti \le u \le \tip} \abs{Y_u}^{m-1})^2 \, \sup_{\ti \le u \le \tip} \abs{Y_u - Y_\tip}^2 \Big] 													\\	
							&\qquad + 3 L_z^2 \hip E\bigg[  \int_\ti^\tip \abs{Z_u - \overline{Z}_\ti}^2 du \bigg]
									+ 3 L_z^2 \hip^2 E\big[ \abs{\overline{Z}_\ti - \widehat{Z}_i}^2 \big] 																																		\\
					&\le 3 L_y^2 \hip^2 E\Big[ \big( 1 + 2 \sup_{\ti \le u \le \tip} \abs{Y_u}^{m-1} \big)^4 \Big]^\half E\Big[ \sup_{\ti \le u \le \tip} \abs{Y_u - Y_\tip}^4 \Big]^\half		\\
							&\qquad + 3 L_z^2 \hip E\bigg[  \int_\ti^\tip \abs{Z_u - \overline{Z}_\ti}^2 du \bigg]
									+ 3 L_z^2 \hip^2 E\big[ \abs{\overline{Z}_\ti - \widehat{Z}_i}^2 \big] .					
			\end{align*}	
		Summing over $i$, we have for some numerical constant $c$
			\begin{align*}
				\sum_{i=0}^{N-1} \frac{ E\big[ \abs{Y_\ti - \widehat{Y}_i}^2 \big] }{\hip}
					&\le 3 L_y^2 \sum_{i=0}^{N-1} \hip  E\Big[ 2^3 (1 + 2^4 \sup_{\ti \le u \le \tip} \abs{Y_u}^{4(m-1)}) \Big]^\half \textrm{REG}_{Y,4}(\hip)^\half			\\
							&\qquad + 3 L_z^2 E\bigg[ \sum_{i=0}^{N-1} \int_\ti^\tip \abs{Z_u - \overline{Z}_\ti}^2 du \bigg]
									+ 3 L_z^2 \sum_{i=0}^{N-1} E\big[ \abs{\overline{Z}_\ti - \widehat{Z}_i}^2 \big] \hip 																					\\
					&\le c L_y^2 \, T \, \big(1+\norm{Y}_{\S^{4(m-1)}}^{4(m-1)}\big)^\half \ \textrm{REG}_{Y,4}(\abs{\pi})^\half
									+ c L_z^2 \ \textrm{REG}_{Z,2}(\abs{\pi})																																					\\
							&\qquad + c L_z^2 \sum_{i=0}^{N-1} E\big[ \abs{\overline{Z}_\ti - \widehat{Z}_i}^2 \big] \hip .
			\end{align*}	
		The first part of the proposition guarantees that the third term is bounded by $C^Z \abs{\pi}$, 
		while \eqref{equation---reference--path.regularity.Y.Z} guarantees that $\textrm{REG}_{Y,4}(\abs{\pi})^\half \le C_{\textrm{PR:Y}} \abs{\pi}$ and 
		$\textrm{REG}_{Z,2}(\abs{\pi}) \le C_{\textrm{PR:Z}} \abs{\pi}$.
	\end{proof}

\bibliographystyle{alpha} 

\bibliography{ATSschemeBiblio} 


\end{document}